\documentclass[reqno,11pt,centertags]{amsart}
\usepackage{amsmath,amsthm,amscd,amssymb,latexsym,upref}
\usepackage[british]{babel}
\usepackage[pdftex]{color}
\usepackage{enumitem}
\usepackage[colorlinks=true,linkcolor=black,citecolor=black]{hyperref}
\usepackage[margin=4cm]{geometry}
\usepackage{graphicx}
\usepackage{mathtools}

\DeclareFontFamily{OT1}{pzc}{}
\DeclareFontShape{OT1}{pzc}{m}{it}{<-> s * [1.10] pzcmi7t}{}
\DeclareMathAlphabet{\mathpzc}{OT1}{pzc}{m}{it}

\newcommand{\re}{\text{\rm Re\,}}
\newcommand{\im}{\text{\rm Im\,}}

\newcommand{\bd}{{\mathbb{D}}}
\newcommand{\bn}{{\mathbb{N}}}
\newcommand{\br}{{\mathbb{R}}}

\newcommand{\bc}{{\mathbb{C}}}

\newcommand{\ch}{{\mathcal{H}}}

\newcommand{\cl}{{\mathcal{L}}}
\newcommand{\css}{{\mathcal{S}}}
\newcommand{\pzch}{{\mathpzc{h}}}
\newcommand{\pzce}{{\mathpzc{e}}}
\newcommand{\pzcq}{{\mathpzc{q}}}
\newcommand{\pzcl}{{\mathpzc{l}}}

\renewcommand{\a}{\alpha}
\renewcommand{\b}{\beta}
\renewcommand{\l}{\lambda}
\newcommand{\s}{\sigma}

\newcommand{\ep}{\varepsilon}

\renewcommand{\d}{\delta}

\newcommand{\g}{\gamma}
\renewcommand{\gg}{\Gamma}
\newcommand{\z}{\zeta}

\newcommand{\bsl}{\backslash}
\newcommand{\ovl}{\overline}

\newcommand{\lb}{\left[}
\newcommand{\rb}{\right]}
\newcommand{\lp}{\left(}
\newcommand{\rp}{\right)}

\newcommand{\sqpm}{\mathrm{sq}_\pm}
\newcommand{\sqp}{\mathrm{sq}_+}
\newcommand{\sqm}{\mathrm{sq}_-}

\DeclarePairedDelimiter{\bra}{(}{)}

\allowdisplaybreaks
\numberwithin{equation}{section}

\newtheorem{theorem}{Theorem}[section]

\newtheorem{lemma}[theorem]{Lemma}
\newtheorem{corollary}[theorem]{Corollary}
\newtheorem{proposition}[theorem]{Proposition}

\theoremstyle{definition}

\newtheorem{remark}[theorem]{Remark}
\newtheorem*{remark*}{Remark}

\newtheorem{example}[theorem]{Example}

\title[Lieb--Thirring and Jensen sums]{Lieb--Thirring and Jensen sums for non-self-adjoint Schr\"odinger operators on the half-line}
\author{Leonid Golinskii}
\email{golinskii@ilt.kharkov.ua}
\address{B. Verkin Institute for Low Temperature Physics and Engineering, Ukrainian
Academy of Sciences, 47 Nauky ave., Kharkiv 61103, Ukraine}
\author{Alexei Stepanenko}
\email{StepanenkoA@cardiff.ac.uk}
\address{School of Mathematics, Cardiff University, Senghennydd Road, Cardiff CF24 4AG, Wales, UK}

\date{\today}

\makeatletter
\@namedef{subjclassname@2020}{%
  \textup{2020} Mathematics Subject Classification}
\makeatother
\subjclass[2020]{47B28, 34L15}

\keywords{ Non-self-adjoint Schr\"odinger operators, discrete spectrum, Jost solutions, Lieb--Thirring type inequality, dissipative barrier potentials}

\begin{document}
\begin{abstract}
  We prove upper and lower bounds for sums of eigenvalues of Lieb--Thirring type for non-self-adjoint Schr\"odinger operators on the half-line.
  The upper bounds are established for general classes of integrable
  potentials and  are shown to be
  optimal in various senses by proving the lower bounds for specific
  potentials.
  We consider sums that correspond to both the critical and non-critical cases. 
\end{abstract}
\maketitle
\section*{Introduction}
There is a vast literature on the spectral theory of self-adjoint Schr\"odinger operators, motivated by their numerous applications in
various areas of mathematical physics. One of the highlights of this theory is  the seminal Lieb--Thirring inequality for operators on $L^2(\br^d)$,
$d\in\bn$, which describes the discrete spectrum of such operators. 
For the case of real line $d=1$ it reads \cite{LiebThirring}
\begin{equation}
  \label{eq:classical-lt}
  \sum_{\lambda \in \sigma_d(H)}|\lambda|^{\mu} \leq C(\mu) \int_{-\infty}^\infty [q_-(x)]^{\mu+1/2} dx, \qquad \mu \geq \frac{1}{2},
\end{equation}
where $C(\mu) > 0$ depends only on $\mu$,  $H$ denotes a Schr\"odinger operator on $\br$ with real-valued potential $q$ and $q_-(x) = \max (0 , - q(x))$.

By comparison, the non-self-adjoint theory is in its youth. The results obtained in the last  two decades have revealed new phenomena and
demonstrated crucial differences between SA and NSA theories. Among the problems which have attracted attention, let us mention spectral enclosure
results and bounds on the number of complex eigenvalues \cite{Davies1,Davies2,LapSaf,Frank1,Frank2,FLSnumber,BogCuen}.
Another active area of interest is non-self-adjoint generalisations of Lieb--Thirring inequalities for Schr\"odinger operators
\cite{FraLapLieb,DHK,FraSab,Safronov,GolKup,Frank3,BogNew},
as well as for other types of operators \cite{DHK2,Sambou,Dubu1,Dubu2,Graphene}.
Still, many questions remain unanswered.

The main object under consideration in the present paper is a Schr\"odinger operator
\begin{equation}
  \label{eq:schrod-H}
  H = H_q :=  - \frac{d^2}{dx^2} + q \qquad \text{on} \qquad L^2(\br_+)
\end{equation}
endowed with a Dirichlet boundary condition at 0, where the potential $q \in L^1(\br_+)$ may be complex-valued.
As is well known, the set of discrete eigenvalues $\sigma_d(H)$ (i.e., eigenvalues of finite algebraic multiplicity in $\bc \backslash \br_+$) may be countably infinite and may accumulate only to $\br_+$.
Lieb--Thirring-type inequalities give information on the distribution of the eigenvalues and, in particular, 
on the rate of accumulation to points in~$\br_+$.

In this paper, we study sums of eigenvalues of the form
\begin{equation}
  \label{eq:defn-S-ep}
  S_\ep(H) := \sum_{\lambda \in \sigma_d(H)} \frac{\text{dist}(\lambda,\br_+)}{|\lambda|^{(1-\ep)/2}}, \qquad \ep \geq 0. 
\end{equation}
Here, eigenvalues of higher algebraic multiplicity are repeated in the sums accordingly. We refer to $S_\ep(H)$ as the {\it Lieb--Thirring sums}.
Note that, in the case when $q$ is real, the eigenvalues of $H_q$ are all negative, so $S_\ep(H_q)$ coincides with the classical
Lieb--Thirring sum in (\ref{eq:classical-lt}), with $\mu = (1 + \ep)/2$. Note also that, by \cite{FLS},  the spectral enclosure $|\lambda|\le\|q\|_1^2$ 
holds for every $\lambda\in\sigma_d(H)$
where, as usual, 
\begin{equation}
  \label{eq:L1-norm}
  \|q\|_1 := \int_0^\infty |q(x)| dx, \qquad q \in L^1(\br_+).
\end{equation}
So, there is a simple relation between the Lieb--Thirring sums with different~$\ep$
\begin{equation}
  \label{twolt}
S_{\ep_2}(H_q)\le \|q\|_1^{\ep_2-\ep_1}\,S_{\ep_1}(H_q), \qquad 0\le\ep_1<\ep_2.
\end{equation}

We also study the sums
\begin{equation}
  \label{eq:defn-J}
  J(H) := \sum_{\lambda \in \sigma_d(H)} \im \sqrt{\lambda} \, , 
\end{equation}
$\sqrt{\cdot}$ denotes the branch of the square root  such that $\im \sqrt{z} > 0$ for all $z \in \bc \bsl \br_+$,
and we refer to $J(H)$ as the {\it Jensen sums}. Notably, $J(H)$ arises naturally from Jensen's formula in complex analysis.
It follows immediately from the inequality \cite[Lemma 1]{DHK}
\begin{equation}\label{dhk09}
|\l|^{1/2}\,|\im\sqrt{\l}|\le {\rm dist} (\l,\br_+)\le 2|\l|^{1/2}\,|\im\sqrt{\l}|, 
\end{equation}
that $J(H)$ is equivalent to $S_0(H)$
\begin{equation}\label{jenvslt} 
J(H) \leq S_0(H) \leq 2 J(H).
\end{equation}

The aim of the paper is two-fold. On one hand, we shall establish upper bounds for the sums $S_\ep(H)$, $\ep \geq 0$, and $J(H)$.
While the upper bounds for the sums  $S_\ep(H)$, $\ep > 0$, (i.e., the non-critical case) hold for arbitrary integrable potentials,
the upper bounds for the sums  $J(H)$ (i.e., the critical case) are only valid for sub-classes of integrable potentials.
On the other hand, corresponding lower bounds shall be proven for specific potentials,
demonstrating optimality of our upper bounds in various senses.
Moreover, in Section 3 we shall construct an integrable potential such that the sum $J(H)  =\infty$.

\subsection*{Summary of main results}

Our analysis is based on identifying the square roots of eigenvalues of the Schr\"odinger operator $H$ \eqref{eq:schrod-H} 
with the zeros of an analytic function in the upper-half of the complex plane $\bc_+$. The idea of using methods of complex analysis in 
the theory of non-self-adjoint Schr\"odinger operator on the half-line goes back to the pioneering papers of Naimark \cite{Naimark} and 
Levin \cite{Levin}, and reaches its culmination in the famous series of papers by Pavlov \cite{Pavlov1,Pavlov2,Pavlov3}, who found the 
threshold between finitely and infinitely many eigenvalues in the case of a complex potential.

Let us first recall the notion of a Jost function, which will be useful for describing the basic ideas of the proofs,
and then proceed to give an account of our main results.
\subsubsection*{Jost functions}

It is well known \cite[Theorems 2.2.1 and 2.3.1]{Naimark} that for any $z \in \bc_+$, the Schr\"odinger equation on $\br_+$
\begin{equation}
  \label{eq:schrod-eq}
  - y'' + q(x) y = z^2 y, \qquad q \in L^1(\br_+)
\end{equation}
has a unique solution $e_+(\cdot,z)$ with the property that $e_+(x,\cdot)$ is analytic on $\bc_+$ for all $x \geq 0$ and  
\begin{equation}
  \label{eq:Jost-asymp}
  e_+(x,z) =e^{ixz}\bigl(1+o(1)\bigr), \quad \text{as} \quad x \to \infty
\end{equation}
uniformly on compact subsets of $\bc_+$. 
$e_+(\cdot,z)$ is referred to as the \emph{Jost solution}.
The \emph{Jost function} is defined as $e_+(z) := e_+(0,z)$, $z \in \bc_+$, and has the property that
\begin{equation}
  \lambda=z^2 \in \sigma_d(H) \qquad \iff \quad e_+(z) = 0. 
\end{equation}
Moreover, the algebraic multiplicity (i.e., the rank of the Riesz projection) of $z^2$ as an eigenvalue of $H$  coincides
with the multiplicity of $z$ as a zero of $e_+$ (see, for instance, \cite[Theorem 5.4 and Lemma 6.2]{GLMZ}).

\subsubsection*{Upper bound for the non-critical case}

Our first result concerns a bound from above for the Lieb-Thirring sums $S_\ep(H)$ in the non-critical case $\ep > 0$. It is valid for Schr\"odinger operators with
arbitrary integrable potentials.

\begin{theorem}[= Theorem \ref{quanbou}]
  \label{th:upper-non-crit-intro}
  For every $\ep > 0$, there exists a constant $K(\ep) > 0$ depending only on $\ep$, such that for any potential $q \in L^1(\br_+)$, we have
  \begin{equation}
    \label{eq:th-0p1}
    S_\ep(H_q) = \sum_{\lambda \in \sigma_d(H_q)} \frac{\mathrm{dist}(\lambda,\br_+)}{|\lambda|^{(1-\ep)/2}} \leq K(\ep) \|q\|_1^{1+\ep}. 
  \end{equation}
\end{theorem}

  Given a pair $(\a, \b)$ of positive parameters, we define a generalised Lieb--Thirring sum $S_{\a,\b}(H_q)$ by \cite{FrankPrivate}
\begin{equation}\label{genltsum}
S_{\a,\b}^{2\a}(H_q) :=\sum_{\lambda \in \sigma_d(H_q)} |\l|^\a\,\left[\frac{\mathrm{dist}(\lambda,\br_+)}{|\lambda|}\right]^\b 
=\sum_{\lambda \in \sigma_d(H_q)} \frac{\mathrm{dist}^\b(\lambda,\br_+)}{|\lambda|^{\b-\a}}\,.
\end{equation}
In terms of such sums, Theorem \ref{th:upper-non-crit-intro} takes the form
\begin{equation}\label{theor0.1}
S_{\a,1}(H_q)\le C_\a\,\|q\|_1, \qquad \forall \a>\frac12.
\end{equation}
We study such generalised Lieb--Thirring sums in more detail in Proposition \ref{generltsup}.

The proof of Theorem \ref{th:upper-non-crit-intro} is based on the application of a result of Borichev, Golinskii and Kupin \cite{BGK1}
concerning the Blaschke-type conditions on zeros of analytic functions on the
unit disk $\bd$ satisfying appropriate growth conditions at the boundary.
An analytic function on $\bd$ is constructed from the Jost function $e_+$ using a certain conformal mapping, and
the growth conditions are verified by applying classical estimates for $e_+$.

\subsubsection*{Upper bounds for the critical case}

Let us address upper bounds for the Jensen sums $J(H)$. We proceed by embarking on a study of sub-classes of $L^1(\br_+)$.

To begin with, we introduce a pair of positive, continuous functions $a$ and $\hat{a}$ on $\br_+$, such that
\begin{equation}
  \label{eq:a-ahat}
  \hat{a}(x) = \frac{x}{a(x)},  \qquad a(x) = \frac{x}{\hat{a}(x)}, \qquad x \in \br_+. 
\end{equation}
We will refer to $a$ and $\hat{a}$ as weight functions. We require that:
\begin{itemize}
\item $a$ is monotonically increasing.
\item $\hat{a}$ is strictly monotonically increasing, $\hat{a}(0) = 0$ and $\hat{a}(\infty) = \infty$.
\end{itemize}

Introduce the norm
\begin{equation}
  \label{eq:a-norm-intro}
  \|q\|_a := \int_0^\infty a(x) |q(x)| dx,
\end{equation}
which agrees with \eqref{eq:L1-norm} for $a\equiv 1$. We consider sub-classes of $L^1(\br_+)$ of the form
\begin{equation}
  \label{eq:Q-a-class-intro}
  Q_a := \{ q \in L^1(\br_+) : \|q\|_a < \infty \}.
\end{equation}
In its most general form, our upper bound for the Jensen sum reads as follows.
\begin{theorem}[= Theorem \ref{ltgente}]
  \label{th:-Jensen-upper-intro}
  Let $a$ and $\hat{a}$ be a pair of weight functions as described above.
  Assume also that
  \begin{equation}
    \label{eq:a-int-cond-intro}
    \int_1^\infty \frac{dx}{xa(x)}<\infty.
  \end{equation}
Then, for each potential $q\in Q_a$ and each $\d\in(0,1)$, we have 
\begin{equation}\label{eq:JH-upper-intro}
J(H_q) \le y\log\frac{1+\d}{(1-\d)^2} + \frac4{\pi}\,\|q\|_a\int_{\frac1{y}}^\infty \frac{dx}{xa(x)},
\end{equation}
where $y=y(\d,a,\|q\|_a)>0$ is uniquely determined by
\begin{equation}\label{eq:testpo-intro}
\hat a\left(\frac1{y}\right)\,\|q\|_a=\log (1+\d).
\end{equation}
\end{theorem}

We emphasise that this upper bound is not applicable for arbitrary potentials $q \in L^1(\br_+)$.
 Loosely speaking, the conditions $\|q\|_a<\infty$ and \eqref{eq:a-int-cond-intro} may contradict each other, as far as the
growth of $a$ goes. An instructive family of integrable potentials is considered in Remark \ref{ex:log-alph}, namely,
\begin{equation}
  \label{eq:slow-log-q-intro}
  q(x)= \frac{i}{x \log^\alpha (x)} \chi_{[e,\infty)}(x), \qquad \alpha > 1,\quad x \in \br_+,
\end{equation}
where $\chi$ denotes the indicator function.
For $\alpha > 2$, there exists an appropriate weight function $a$, and Theorem \ref{th:-Jensen-upper-intro} is applicable to $q$.
For $1 < \alpha \leq 2 $, such a weight function $a$ does not exist.

We do not claim that $J(H_q)=\infty$ for the potentials $q$ in \eqref{eq:slow-log-q-intro} with $1<\alpha\leq 2$.  In Theorem \ref{mainth},
we construct an example of a potential for which the Jensen sum diverges, showing that Theorem \ref{th:-Jensen-upper-intro} cannot be extended 
to all integrable potentials.

Theorem \ref{th:-Jensen-upper-intro} is applied to obtain upper bounds for $J(H)$ valid for two important specific classes of potentials.
\begin{enumerate}[label=(\alph*),wide,labelindent=5pt]
\item[(A)] (See Corollary \ref{col:poly-bound})  \emph{ Let $p \in (0,1)$ and $a(x) = 1+ x^p$. Then for each potential $q \in Q_a$, we have}
  \begin{equation}\label{eq:poly-bound-intro}
    J(H_q) \leq \tfrac{4}{\pi} \|q\|_a \log \lp 1 + \|q\|_a \rp  + \tfrac{9}{p}\|q\|_a + 2.
  \end{equation}
\end{enumerate}
In \cite{Safronov}, Safronov has also obtained a bound for the Jensen sum $J(H)$, valid for potentials $q \in L^1(\br_+)$ 
satisfying $\|x^p q\|_1 < \infty$ for some $p \in (0,1)$.
Comparatively, the above result (A) offers an improved asymptotic estimate for semiclassical Schr\"odinger operators (see Remark \ref{rem:saf}). 
\begin{enumerate}[label=(\alph*),wide,labelindent=5pt]
\item[(B)] (See Corollary \ref{compsup})
 \emph{Suppose the potential $q \in L^1(\br_+)$ is compactly supported. Then, for every  $R>1$ with ${\rm supp}(q)\subset [0,R]$, we have }
\begin{equation}\label{eq:comp-upper-intro}
J(H_q)\le 7\left[\frac1{R}+\|q\|_1\Bigl(1+\log(1+\|q\|_1)+\log R\Bigr)\right].
\end{equation}
\end{enumerate}
As we will see below, this bound is optimal in a certain asymptotic sense.

The proof of Theorem \ref{th:-Jensen-upper-intro} centers around establishing improved estimates for the Jost function $e_+$ corresponding
to potentials in a given sub-class $Q_a$.
These improved estimates are obtained by combining the arguments for the classical case with the following simple principle:
\begin{equation}\label{principle}
  0 < A \leq \min(X_1,\, X_2) \ \Rightarrow \ A = a(A)\hat{a}(A) \leq a(X_1)\hat{a}(X_2).
\end{equation}
The bound (\ref{eq:JH-upper-intro}) of Theorem \ref{th:-Jensen-upper-intro}
is proven by using these improved estimates for $e_+$ in conjunction with Jensen's formula.
The proofs of Corollaries~\ref{col:poly-bound} and \ref{compsup} amount to appropriate choices for $a$ and $\d$.

\subsubsection*{Lower bounds for dissipative barrier potentials}

The optimality of the above upper bounds can be addressed by studying corresponding lower bounds for
Schr\"odinger operators with so-called {\it dissipative barrier potentials}.
Precisely, for $\g,R > 0$, we consider the Schr\"odinger operator
\begin{equation}
  \label{eq:db-defn-intro}
  L_{\g,R} :=  - \frac{d^2}{dx^2} + i \gamma \chi_{[0,R]} \qquad \text{on} \qquad L^2(\br_+)
\end{equation}
endowed with a Dirichlet boundary condition at 0. The dissipative barrier potentials find applications in the numerical computation of eigenvalues, 
where they are considered as a perturbation of a fixed background potential \cite{MarlSch,StepaIncl}. We focus on establishing our estimates 
for large enough $R$. Observe that $\|i \gamma \chi_{[0,R]}\|_1 = \gamma R$.
\begin{theorem}[= Theorem \ref{litrbelow}]
  \label{th:db-lower-intro}
  Suppose that $R \geq 600(\gamma^{3/4}+\gamma^{-3/4})$.
  
  $(i)$    We have the following lower bound
    \begin{equation}
      \label{eq:crit-lower-db}
    2 J(L_{\g,R})\geq S_0(L_{\g,R}) \geq \frac{ \gamma R}{16 \pi} \log R.
  \end{equation}
  
 $(ii)$  Let $0<\ep<1$. Under the stronger assumption on $R$
\begin{equation}\label{ranger1-intro}
R\ge \frac4{e^2\g}(64\pi)^{2/\ep}+1,
\end{equation}
 we have the lower bound
  \begin{equation}
    \label{eq:non-crit-lower-db}
    S_\ep(L_{\g,R}) \geq \frac1{256\pi\ep}\,\frac{(\g R)^{1+\ep}}{\log^\ep R}\,.
\end{equation}

\end{theorem}
 The estimate (\ref{eq:crit-lower-db}) shows that
$$ \sup_{0\not=q\in L^1(\br_+)} \frac{S_0(H_q)}{\|q\|_1}=+\infty. $$
An analogous, but slightly less explicit, result for Schr\"odinger operators on the whole real line has appeared in \cite{BogStam} (cf. Remark \ref{rem:result-S-B}). Notably, our proofs seem to use rather different methods.

The  main ideas in the proof of Theorem \ref{th:db-lower-intro} are as follows.
Starting from the Jost function of $L_{\g,R}$, we construct a countable family of equations,  each of which is in the form of a fixed point equation.
We are able to use the contraction mapping principle to prove that each equation has a unique solution
 corresponding to exactly one zero of the Jost function $e_+$  (or, more precisely, one zero of the analytic continuation of $e_+$ to $\bc$).

As it turns out, each equation has a convenient form that allows us to gain  quantitative information about its solution, 
hence about an individual zero of $e_+$.
Estimates for the different equations can be combined to obtain lower bounds for the sums $J(L_{\g,R})$ and $S_\ep(L_{\g,R})$ 
as well as other quantities, such as the number of eigenvalues (see Corollary \ref{col:number}).

Finally, note that, when applied to the Schr\"odinger operators $L_{\g,R}$ \eqref{eq:db-defn-intro}, the upper bound \eqref{eq:comp-upper-intro} gives 
the optimal asymptotic estimate (see Proposition~\ref{prop:two-sided})
\begin{equation}
  \label{eq:db-optimal}
  J(L_{\g,R}) = O(R \log R), \quad \text{as} \quad R \to \infty.
\end{equation}

\subsubsection*{Divergent Jensen sum}
As mentioned, while Theorem \ref{th:-Jensen-upper-intro} provides an upper bound for $J(H)$ for a wide range of potentials,
there exist integrable potentials to which it does not apply. It is therefore natural to ask whether or  not it is possible to
extend this upper bound to arbitrary integrable potentials. Our final result show that this is impossible.   

\begin{theorem}[=Theorem \ref{mainth}]
  \label{Jensen-infty}
  There exists a potential $q \in L^1(\br_+)$ such that $J(H_q) = \infty$. 
\end{theorem}

The proof of this result uses two crucial ingredients. The first is an idea of B\"ogli \cite{BogAcc}, which allows one to construct a Schr\"odinger operator
whose eigenvalues approximate the union of the eigenvalues of a given sequence of Schr\"odinger operators $\cl_n$, $n \in \bn$. 
The second is the lower bound of Theorem \ref{th:db-lower-intro} for the Jensen sum $J(L_{\gamma,R})$.
Indeed, the given sequence of Schr\"odinger operators $\cl_n$ in our case shall have dissipative barrier potentials.
Note that the explicit condition $R \geq 600(\gamma^{3/4} + \gamma^{-3/4})$ in Theorem \ref{th:db-lower-intro} plays an important role in
Theorem \ref{Jensen-infty}.

\begin{remark*} ($\br_+$ vs $\br$).
Given a potential $q\in L^1(\mathbb{R}_+)$, denote by $Q$ its even extension on the whole line. 
By Proposition \ref{evenext} below, there is inclusion $\sigma_d(H_q)\subset\sigma_d(H_Q)$, counting multiplicities,
for the discrete spectra of Schr\"odinger--Dirichlet operator $H_q$ on $L^2(\br_+)$ and Schr\"odinger operator $H_Q$ on $L^2(\br)$.
Hence, the inequality
\begin{equation}
  \label{eq:R-vs-R+}
\sum_{\l\in\s_d(H_q)} \Phi(\l) \leq \sum_{\l\in\s_d(\ch_Q)} \Phi(\l), \qquad q\in L^1(\br_+),
\end{equation} 
holds with an arbitrary nonnegative function $\Phi$ on the complex plane. Thereby, upper bounds, such as \eqref{eq:th-0p1}, for $H_q$ can be derived 
from the corresponding results for the operator $H_Q$. As an example, the spectral enclosure \cite{FLS} mentioned above is a direct consequence of
the result for the whole line \cite[Theorem 4]{Davies1}.

Several inequalities of Lieb--Thirring-type for Schr\"odinger operators with complex potentials on $L^2(\br)$ are known nowadays, but neither covers
completely the main results of the paper. The result of Frank and Sabin \cite[Theorem 16]{FraSab} in dimension one is \eqref{eq:th-0p1} with $\ep>1$. 
The case $\ep=1$ is a consequence of \cite[Theorem 1.3]{Frank3}. The result of Demuth, Hansmann and Katriel \cite[Corollary 3]{DHK} in dimension one reads
\begin{equation*}
\sum_{\lambda \in \sigma_d(\ch_Q)} \frac{\mathrm{dist}^{p+\ep}(\lambda,\br_+)}{|\lambda|^{\frac12+\ep}} \leq C(p,\ep) \|Q\|_{L^p(\br)}^p, \qquad p\ge\frac32, \quad \ep \in (0,1).
\end{equation*}
Recently, B\"ogli \cite{BogNew} has extended this result considerably by including a much wider class of sums.
The results of both DHK and B\"ogli are not applicable for arbitrary $L^1$ potentials, hence do not imply Theorem \ref{th:upper-non-crit-intro}.

We believe that the results for Schr\"odinger operators with complex potentials on $L^2(\br)$, analogous to our upper bounds,
can be obtained along the same line of reasoning by using similar methods. The study of this problem should be carried out elsewhere. 
\end{remark*}

\subsubsection*{Outline of the paper}
 In Section 1, we focus on upper bounds for the Lieb--Thirring sums with an arbitrary potential $q\in L^1(\br_+)$,
and for the Jensen sums with potentials $q\in Q_a$. Section 2 is devoted to the spectral analysis of Schr\"odinger operators with dissipative barrier
potentials and to the lower bounds for the Lieb-Thirring and Jensen sums with such potentials.  In Section 3 we prove Theorem~\ref{Jensen-infty}.

\subsection*{Acknowledgements} The authors thank S. B\"ogli and J.-C. Cuenin for helpful discussions
and R. Frank for enlightening comments which motivated us to include Proposition \ref{generltsup}. 
AS is supported by an EPSRC studentship EP/R513003/1 and thanks his PhD supervisors  M. Marletta and J. Ben-Artzi for helpful discussions and guidance.

\section{Classes of potentials and inequalities for sums of eigenvalues}

As we mentioned earlier in the introduction, a complex number $\z \in \bc_+$ belongs to the zero set $Z(e_+)$ of the Jost function 
if and only if $\lambda=\z^2\in\s_d(H)$, and the zero multiplicity coincides with the algebraic multiplicity of the corresponding eigenvalue.   
Therefore, the divisor $Z(e_+)$ (zeros counting multiplicities) has a precise spectral interpretation.
In this section, we study this divisor using various results from complex analysis and hence obtain bounds for sums of 
Lieb-Thirring and Jensen types. Throughout the section, we shall let
\begin{equation*}
  \bc_+^0:=\{z \in \bc : \im z\ge0, \ z\not=0\}.
\end{equation*}

\subsection{Bounds for Lieb--Thirring sums}
Recall that the Lieb--Thirring sum for a Schr\"odinger--Dirichlet operator $H$ is given by
$$ S_\ep(H)=\sum_{\l\in\s_d(H)} \frac{{\rm dist}(\l,\br_+)}{|\l|^{\frac{1-\ep}2}}\,, \qquad 0\le\ep<1. $$
Our first result gives an upper bound for $S_\ep(H)$ in the non-critical case of $\ep>0$ and arbitrary $q \in L^1(\br_+)$.

\begin{theorem}[= Theorem \ref{th:upper-non-crit-intro}]\label{quanbou}
For every $\ep>0$, there exists a constant $K(\ep)>~0$, depending only on $\ep$, such that
\begin{equation}\label{ltcla}
S_\ep(H_q)\le K(\ep)\|q\|_1^{1+\ep}.
\end{equation}
\end{theorem}
\begin{proof}
A key ingredient of the proof is the following well-known inequality for the Jost function (see, e.g., \cite[Lemma 1]{Stepin})
\begin{equation}\label{ineqjost}
|e_+(z)-1|\le \exp\left\{\frac{\|q\|_1}{|z|}\right\}-1, \qquad z\in\bc_+^0.
\end{equation}

Let
$$ y:=\frac{\|q\|_1}{\kappa}>0, \qquad \kappa:=\log\frac32. $$ 
By \eqref{ineqjost}, 
$$ |e_+(iy)-1|\le\frac12, \qquad |e_+(iy)|\ge\frac12. $$

Consider the function
$$ g(z):=\frac{e_+(yz)}{e_+(iy)}\,, \quad z\in\bc_+, \qquad g(i)=1. $$
By the definition of $y$, we have
\begin{equation*}
\begin{split}
|g(z)| &\le 2|e_+(yz)|\le 2\exp\left\{\frac{\|q\|_1}{y|z|}\right\}=2\exp\left\{\frac{\kappa}{|z|}\right\}, \\
\log|g(z)| &\le \log 2+\frac{\kappa}{|z|}< \log 2\,\frac{1+|z|}{|z|}\,.
\end{split}
\end{equation*}

To go over to the unit disk, we introduce a new variable,
\begin{equation}\label{eq:var-w}
  w=w(z)=\frac{z-i}{z+i}\,:\,\bc_+ \ \rightarrow \ \bd, \qquad z=z(w)=i\,\frac{1+w}{1-w}\,.
\end{equation}
Write $f(w):=g(z(w))$. An elementary inequality
$$ \frac2{1+|z|}\le |1-w(z)|\le\frac{2\sqrt2}{1+|z|}, \qquad z \in \bc_+, $$
gives the following bound for $f$
\begin{equation}\label{bgkass} 
\log|f(w)|\le \frac{2 \sqrt{2}\log 2}{|1+w|}\,, \qquad f(0)=1. 
\end{equation}

The Blaschke-type conditions for zeros of such analytic functions in $\bd$ are obtained in \cite{BGK1} (see \cite{BGK2} for some advances)
$$ \sum_{\eta\in Z(f)} (1-|\eta|)|1+\eta|^\ep\le K_1(\ep), \qquad \forall\ep>0, $$
where $K_1(\ep)>0$ depends only on $\ep$. Going back to the upper half-plane and using another elementary inequality
\begin{equation}\label{1eq:elem-ineq-z-w}
  \frac{\im z}{1+|z|^2}\le 1-|w|\le \frac{8\,\im z}{1+|z|^2}\,,
\end{equation}
we come to the following relation for the divisor $Z(g)$
\begin{equation*}
\sum_{\xi\in Z(g)} \frac{\im\xi}{1+|\xi|^2}\,\frac{|\xi|^\ep}{|\xi+i|^\ep}\le K_2(\ep).
\end{equation*}
But $\xi\in Z(g)$ is equivalent to $\z=y\xi\in Z(e_+)$, so
$$ \left(\frac{\kappa}{\|q\|_1}\right)^{1+\ep}\,
\sum_{\z\in Z(e_+)}\frac{\im\z\,|\z|^\ep}{\left\{1+\left(\frac{\kappa|\z|}{\|q\|_1}\right)^2\right\}\,
\left|\frac{\kappa \z}{\|q\|_1}+i\right|^\ep}\le K_2(\ep). $$ 
 The aforementioned spectral enclosure result ensures that $|\z|\le \|q\|_1$ for $\z\in Z(e_+)$. It follows that both factors in the 
denominator are bounded from above by some constants depending only on $\ep$. We come to
\begin{equation}\label{ltdivis}
\sum_{\z\in Z(e_+)} (\im\z)\,|\z|^\ep\le K(\ep)\|q\|_1^{1+\ep},
\end{equation}
where a positive constant $K$ depends only on $\ep$.

 To complete the proof, we employ the inequality \eqref{dhk09}, mentioned in the introduction. So, \eqref{ltcla} follows.
\end{proof}

\subsection{Classes of potentials and Jensen sums}

In the rest of the section, we study the behavior of the discrete spectrum for Schr\"odinger operators within
special classes of potentials.

Let $a$ be a monotonically increasing and locally integrable, nonnegative function on $\br_+$. 
Consider the classes of complex-valued potentials 
\begin{equation}
  \label{eq:Q-a}
Q_a:=\{q \in L^1(\br_+):\ \int_0^\infty a(x)|q(x)|dx<\infty \}.
\end{equation}
The weight function $a$ is fixed in the sequel, and dependence of constants on $a$ is sometimes omitted.

Define a function $\hat{a}$ on $\br_+$ by
\begin{equation*}
\hat{a}(x):=\frac{x}{a(x)}\,, \qquad x\in\br_+,
\end{equation*}
and put
\begin{equation*}
\omega_a(x,z):=\hat a\left(\frac1{|z|}\right)\,\int_x^\infty a(t)|q(t)|dt, \qquad x\in\br_+, \quad z\in\bc_+^0.
\end{equation*}

\begin{proposition}\label{jost}
Assume that both $a$ and $\hat a$ are monotonically increasing functions on $\br_+$. Then the Jost solution admits the bound
\begin{equation}\label{jostbound}
|e^{-izx}e_+(x,z)-1|\le \exp\left(\omega_a(x,z)\right)-1, \qquad x\in\br_+, \quad z\in\bc_+^0.
\end{equation}
\end{proposition}
\begin{proof}
We follow the arguments of M.A. Naimark for the classical case $a\equiv 1$. 

The Jost solution is known to satisfy the Schr\"odinger integral equation
\begin{equation*}
e_+(x,z)=e^{ixz}+\int_x^\infty \frac{\sin((t-x)z)}{z}q(t)e_+(t,z)dt.
\end{equation*}
The latter can be resolved by the successive approximations method.

Introduce a new unknown function
$$ f(x,z):=e^{-ixz}e_+(x,z)-1, $$
which satisfies
\begin{equation}\label{newvar}
\begin{split}
f(x,z) &= g(x,z)+\int_x^\infty k(t-x,z)q(t)f(t,z)dt, \\
 k(u,z)&:=\frac{\sin uz}{z}\,e^{iuz}, \quad g(x,z) :=\int_x^\infty k(t-x,z)q(t)dt.
\end{split}
\end{equation}
Let
$$ f_1(x,z):=g(x,z), \quad f_{n+1}(x,z)= \int_x^\infty k(t-x,z)q(t)f_n(t,z)dt, \quad n\in\bn. $$

In view of an elementary bound for the kernel $k$
$$ |k(u,z)|\le \min\left(u, \frac1{|z|}\right), $$
and monotonicity of $a$ and $\hat a$, we see that
\begin{equation}\label{monot}
|k(u,z)|=\hat a(|k(u,z)|)\,a(|k(u,z)|)\le \hat a\left(\frac1{|z|}\right)\,a(u),
\end{equation}
 cf. \eqref{principle}.

We first estimate $f_1$. By \eqref{monot},
$$ |f_1(x,z)|\le \int_x^\infty |k(t-x,z)||q(t)|dt\le \hat a\left(\frac1{|z|}\right)\,\int_x^\infty a(t-x)|q(t)|dt\le \omega_a(x,z). $$
Assume for induction that
\begin{equation}\label{induc}
|f_j(x,z)|\le \frac{\omega_a^j(x,z)}{j!}, \qquad j=1,2,\ldots,n.
\end{equation}
We compute
\begin{equation*}
\begin{split}
\frac{d}{dx}\Bigl[\omega_a^{n+1}(x,z)\Bigr] &=(n+1)\,\omega_a^n(x,z)\,\frac{d}{dx}\Bigl[\omega_a(x,z)\Bigr] \\
&=-(n+1)\,\omega_a^n(x,z)\,\hat a\left(\frac1{|z|}\right)\,a(x)|q(x)|,
\end{split}
\end{equation*} 
and so
\begin{equation*}
\begin{split}
|f_{n+1}(x,z)| &\le \int_x^\infty |k(t-x,z)||q(t)|\frac{\omega_a^n(t,z)}{n!}dt \\
&\le \frac1{n!}\,\hat a\left(\frac1{|z|}\right)\,\int_x^\infty a(t)|q(t)|\,\omega_a^n(t,z)dt \\
&= -\frac1{(n+1)!}\,\int_x^\infty\, \frac{d}{dt}\Bigl[\omega_a^{n+1}(t,z)\Bigr]\,dt=\frac{\omega_a^{n+1}(x,z)}{(n+1)!}\,.
\end{split}
\end{equation*}
Hence, \eqref{induc} indeed holds for all $n\in\bn$.

It follows that the solution $f$ to \eqref{newvar}, which is known to be unique, satisfies
$$ |f(x,z)|\le \sum_{n=1}^\infty |f_n(x,z)|\le \exp\bigl(\omega_a(x,z)\Bigr)-1 $$
(the latter series converges absolutely and uniformly on the compact subsets of $(x\in\br_+, z\in\bc_+^0)$). 
The bound \eqref{jostbound} follows.
\end{proof}

The above result for $a(x)=x^\alpha$, $\alpha\in[0,1]$, is due to Stepin \cite[Lemma~1]{Stepin}.
The bound for the Jost function $e_+(z)=e_+(0,z)$ is \eqref{jostbound} with $x=0$:
\begin{equation}\label{jostfunc}
|e_+(z)-1|\le \exp\left\{\hat a\left(\frac1{|z|}\right)\,\|q\|_a\right\}-1, \quad \|q\|_a:=\int_0^\infty a(t)|q(t)|dt.
\end{equation}

The following spectral enclosure result is a simple consequence of \eqref{jostfunc} and the basic property of zeros of $e_+$.

\begin{corollary}
Under the hypothesis of Proposition \ref{jost}, define the value
\begin{equation*}
\rho=\rho(a,q):=\inf\Bigl\{t>0: \ \hat a(\sqrt{t})\ge\frac{\log 2}{\|q\|_a}\Bigr\}.
\end{equation*}
Then the discrete spectrum $\sigma_d(H_q)$ is contained in the closed disk 
$$ \sigma_d(H_q)\subset B(0,\rho^{-1}). $$
The case $\hat a(\infty)<\log 2\,\|q\|_a^{-1}$ implies that $\rho=\infty$, and so the discrete spectrum is empty.

\end{corollary}

As a matter of fact, in view of \cite{FLS}, we have a more precise inclusion
\begin{equation}
\sigma_d(H_q)\subset B(0,r), \qquad r:=\min(\rho^{-1}, \,\|q\|^2_1).
\end{equation}

To study the distribution of eigenvalues of $H$ for potentials from the class $Q_a$, we apply standard tools from 
complex analysis (the Jensen formula).
Recall that the Jensen sum is given by
\begin{equation}\label{jensum}
J(H)=\sum_{\lambda\in\sigma_d(H)} \im\sqrt{\lambda}.
\end{equation}
Here $\sqrt{\cdot}=\sqp(\cdot)$ is the branch of the square root, which maps $\bc\bsl\br_+$ onto the upper half-plane $\bc_+$.

\begin{theorem}[= Theorem \ref{th:-Jensen-upper-intro}]\label{ltgente}
In addition to the hypothesis of Proposition~\ref{jost}, assume that
\begin{enumerate}
\item $\hat a$ is a continuous, strictly monotonically increasing function, and
 $\hat a(0)=0$, \  $\hat a(\infty)=\infty$,

\item 
$$ \int_1^\infty \frac{dx}{xa(x)}<\infty. $$
\end{enumerate}
Then, for each potential $q\in Q_a$, and each $\delta\in(0,1)$, the following bound for the Jensen sum holds
\begin{equation}\label{ltgen}
J(H_q)\le y\log\frac{1+\delta}{(1-\delta)^2} + \frac4{\pi}\,\|q\|_a\int_{\frac1{y}}^\infty \frac{dx}{xa(x)},
\end{equation}
where $y=y(\delta,a,\|q\|_a)>0$ is uniquely determined by
\begin{equation}\label{testpo}
\hat a\left(\frac1{y}\right)\,\|q\|_a=\log (1+\delta).
\end{equation}
\end{theorem}
\begin{proof}
The argument is similar to that in Theorem \ref{quanbou}. It follows from \eqref{jostfunc} and \eqref{testpo} that
\begin{equation*}
|e_+(iy)-1|\le 1+\delta-1=\delta, \qquad |e_+(iy)|\ge 1-\delta,
\end{equation*}
so the normalized function
$$ g(z):=\frac{e_+(y z)}{e_+(iy)}\,, \qquad g(i)=1, $$
satisfies
\begin{equation*}
\log|g(z)|\le \log\frac1{1-\delta}+\hat a\left(\frac1{y|z|}\right)\|q\|_a, \qquad z\in\bc_+.
\end{equation*}

Introduce a new variable $w \in \bd$, related to $z \in \bc_+$ by \eqref{eq:var-w}.
For $f(w):=g(z(w))$ one has, as above, $f(0)=1$ and
\begin{equation*}
\log |f(w)|\leq\log \frac1{1-\delta}+\hat a\left(\frac1{y}\left|\frac{1-w}{1+w}\right|\right)\|q\|_a, \quad w\in\bd.
\end{equation*}
For $w=re^{i\theta}$, $|\theta|\le\pi$, it is easy to calculate
\begin{equation*} 
\max_{0\le r\le 1} \left|\frac{1-re^{i\theta}}{1+re^{i\theta}}\right|=
\begin{cases} 1,\ & |\theta|\le \frac{\pi}2, \\
      |\tan\frac{\theta}2\Bigr|, \ & \frac{\pi}2<|\theta|< \pi,
\end{cases} 
\end{equation*}      
so
\begin{equation*}
\log|f(w)|\le
\begin{cases} \log\frac{1+\delta}{1-\delta},\ & |\theta|\le \frac{\pi}2, \\
\log\frac{1}{1-\delta}+\hat a\left(\frac1{y}\Bigl|\tan\frac{\theta}2\Bigr|\right)\|q\|_a, \ & \frac{\pi}2<|\theta|< \pi.
\end{cases}
\end{equation*}

In view of assumption $(2)$, the Jensen formula provides
\begin{equation*}
\begin{split}
\sum_{\eta\in Z(f)} (1-|\eta|) &\le \sum_{\eta\in Z(f)} \log\frac1{|\eta|}  \\ 
&\le \frac12\log\frac{1+\delta}{(1-\delta)^2} +\frac{\|q\|_a}{\pi}\int_{\pi/2}^\pi \hat a\left(\frac1{y}\Bigl(\tan\frac{\theta}2\Bigr)\right)\,d\theta \\
&=\frac12\log\frac{1+\delta}{(1-\delta)^2}+\frac{2\|q\|_a}{\pi}\int_1^\infty\frac{\hat a(y^{-1}t)}{1+t^2}dt, 
\end{split}
\end{equation*}
and hence
\begin{equation*}
\begin{split}
\sum_{\eta\in Z(f)} (1-|\eta|) &\le \frac12\log\frac{1+\delta}{(1-\delta)^2}+\frac{2\|q\|_a}{\pi}\,\int_1^\infty\frac{\hat a(y^{-1}t)}{t^2}dt \\
&\le \frac12\log\frac{1+\delta}{(1-\delta)^2}+\frac{2\|q\|_a}{\pi y}\,\int_{\frac1{y}}^\infty \frac{dx}{xa(x)}=:B.
\end{split}
\end{equation*}

Going back to the function $g$ and the upper half-plane and using \eqref{1eq:elem-ineq-z-w}, we come to
$$ \sum_{\xi\in Z(g)}\frac{\im \xi}{1+|\xi|^2}\le B. $$
The relation between $Z(g)$ and $Z(e_+)$ is straightforward
$$ \xi\in Z(g) \ \Leftrightarrow \ \zeta=y\xi\in Z(e_+), $$
and, hence,
\begin{equation}\label{bla1} 
\sum_{\zeta\in Z(e_+)}\frac{\im \zeta}{1+\Bigl|\frac{\zeta}{y}\Bigr|^2}\le By. 
\end{equation}

As it follows from \eqref{jostfunc},
$$ \hat a\left(\frac1{|z|}\right)\|q\|_a<\log 2 \ \Rightarrow \ e_+(z)\not=0. $$
Therefore,
$$ \hat a\left(\frac1{|\zeta|}\right)\|q\|_a\ge\log 2, \qquad \zeta\in Z(e_+), $$
and so (see the choice of $y$ \eqref{testpo}), by monotonicity of $\hat a$,
$$ \hat a\left(\frac1{|\zeta|}\right)\|q\|_a>\hat a\left(\frac1{y}\right)\,\|q\|_a \ \Rightarrow \ \Bigl|\frac{\zeta}{y}\Bigr|<1. $$

We conclude from \eqref{bla1}, that
\begin{equation*}
\sum_{\zeta\in Z(e_+)} \im{\zeta}\le 2By,
\end{equation*}
and \eqref{ltgen} follows. The proof is complete.
\end{proof}

As a first application of the above result, we study Schr\"odinger operators $H_q$ with potentials $q$ satisfying $\|(1 + x^p)q \|_1 < \infty$ for some $p \in (0,1)$.
Taking $a(x) := x^p$ and any fixed $\delta \in (0,1)$ (e.g., $\delta = 1/2$) in Theorem \ref{ltgente} easily yields the inequality
\begin{equation*}
J(H_q) \leq C(p) \lp \int_0^\infty x^p |q(x)| dx \rp^{\frac{1}{1 - p}}, \qquad p \in (0,1).  
\end{equation*}
The following corollary of Theorem \ref{ltgente} offers a refinement of this bound.
\begin{corollary}
  \label{col:poly-bound}
  Let $p \in (0,1)$ and $a(x) = 1+ x^p$. Then for each potential $q \in Q_a$, the following inequality holds
  \begin{equation}\label{eq:poly-bound}
    J(H_q) \leq \tfrac{4}{\pi}\, \|q\|_a \log \lp 1 + \|q\|_a \rp  + \tfrac{9}{p}\,\|q\|_a + 2.
  \end{equation}
\end{corollary}

\begin{proof}
  Put
  \begin{equation*}
    \delta: = \exp\Bigl(\min \lp \tfrac{1}{2} \|q\|_a, \kappa\rp\Bigr) - 1 \in \bigl(0,\tfrac{1}{2}\bigr], \qquad \kappa=\log\tfrac{3}{2}.
  \end{equation*}
  Then, by \eqref{testpo},
  \begin{equation*}
    A_0 := \frac{\log(1+\delta)}{\|q\|_a}=\hat a\left(\frac1{y}\right) \leq \frac{1}{2} \quad \text{and} 
    \quad \log  \frac{1+\delta}{(1 - \delta)^2} \leq \log 6. 
  \end{equation*}
  
  Since $\hat{a}$ is monotonically increasing, with $\hat{a}(1) = \tfrac{1}{2}$, we must have $y \geq 1$.
  In particular, this implies that
  \begin{equation*}
    \frac{y^{-1}}{1 + y^{-1}} \geq \frac{y^{-1}}{1 + y^{-p}}= \hat{a}(y^{-1}) = A_0, \qquad \frac1{y} \geq \frac{A_0}{1 - A_0}\,,
  \end{equation*}
  and so
  \begin{equation}\label{valuey}
  1\le y\le \frac{1-A_0}{A_0}\le\frac1{A_0}\,.
  \end{equation}
  
  If $\|q\|_a \geq 2 \kappa$, then $\delta = \tfrac{1}{2}$, so $y \leq 3 \|q\|_a$. On the other hand, if $\|q\|_a < 2 \kappa$, 
  then $A_0 = \frac{1}{2}$, so $y=1$ ($\hat a$ is strictly monotonically increasing).  We conclude that
  \begin{equation}\label{eq:y-upper-poly}
    y \leq 3 \|q \|_a + 1. 
  \end{equation}

  The right hand side of (\ref{ltgen}) is the sum of two terms.
  We bound the first one as
  \begin{equation*}
    A_1 := y \log \frac{1+\delta}{(1 - \delta)^2} \leq \log 6 \lp 3 \|q\|_a + 1\rp < 6 \|q\|_a + 2.
  \end{equation*}
  The second (integral) term reads
  \begin{equation*}
    A_2 := \frac{4}{\pi}\, \|q\|_a \int_{1/y}^\infty \frac{d x}{x (1 + x^p)}.
  \end{equation*}
  The integral may be computed, and bounded above, as
  \begin{align*}
    \int_{1/y}^\infty \frac{dx}{x(1 + x^p)} &  = \frac{1}{p}\log \lp 1 + \frac{1}{y^p} \rp + \log y \\
                        & \leq \frac{1}{p y^p}+\log y.  
  \end{align*}
  Using the upper bound (\ref{eq:y-upper-poly}) and the lower bound \eqref{valuey} for $y$, we obtain
  \begin{align*}
    A_2 & \leq \frac{4}{\pi}\, \|q\|_a \lb \log (1 + \|q\|_a) + \log 3 + \frac{1}{p} \rb \\
    & \leq\frac{4}{\pi}\, \|q\|_a \log (1 + \|q\|_a) + \frac{3}{p}\, \|q\|_a.
  \end{align*}
  The bound (\ref{eq:poly-bound}) follows by combining the bounds for $A_1$ and $A_2$.
\end{proof}

\begin{remark}\label{rem:saf}
  In \cite{Safronov}, Safronov also studies Schr\"odinger operators $H_q$ on $\br_+$ with potentials $q$ satisfying $\|(1+x^p) q\|_1<\infty$ for some $p \in (0,1)$,
  and obtains the estimate
  \begin{equation}
    \label{eq:safronov}
    J(H_q) \leq C(p) \lp \int_0^\infty x^p |q(x)| dx \lp \int_0^\infty |q(x)| dx \rp^p  + \int_0^\infty |q(x)| dx  \rp.
  \end{equation}
  Consider the following Schr\"odinger-Dirichlet operators on $\br_+$,
  \begin{equation*}
    H_h = - \frac{d^2}{d x^2} + q(x h),  \qquad h > 0, 
  \end{equation*}
  where $q \in L^1(\br_+)$ is fixed. A rescaling shows that $h \to 0$ is equivalent to a semiclassical limit.
  It can be seen that Corollary \ref{col:poly-bound} gives
  \begin{equation*}
    J(H_h) = O(h^{-(1+p)} \log(\tfrac{1}{h})) \quad \text{as} \quad h \to 0,
  \end{equation*}
  while the estimate (\ref{eq:safronov}) gives
  \begin{equation*}
    J(H_h) = O(h^{-(1 + 2p)}) \quad \text{as} \quad h \to 0,
  \end{equation*}
  hence our result offers an improved asymptotic estimate for  $H_h$.
  
\end{remark}

The next example is more delicate. It presents an integrable potential $q$ that is not covered by Theorem \ref{ltgente}.
More precisely, $q\notin Q_a$ for any weight function $a$ satisfying the assumptions of Theorem \ref{ltgente}.

\begin{example}
  \label{ex:log-alph}
Take $\a>1$ and put
\begin{equation}\label{examp}
    q(x) := \begin{cases} \frac{i}{x\log^\a x},\ & x\ge e, \\
      0, \ & 0<x<e,
    \end{cases}
  \end{equation}
Then, $q \in L^1(\br_+)$. We distinguish two cases.

1. Assume that $\a>2$. Choose $\b$ from $1<\b<\a-1$ and denote
\begin{equation*}
    a(x) := \begin{cases} \log^\b x, \  & x\ge e^\b, \\
      \b^\b, \ & 0<x<e^\b,
    \end{cases}
  \end{equation*}
so $a$ is a positive, monotonically increasing and continuous function on $\br_+$. Then,
\begin{equation*}
   \hat a(x) = \begin{cases} \frac{x}{\log^\b x}, \  & x\ge e^\b, \\
      \b^{-\b}x, \ & 0<x<e^\b.
    \end{cases}
  \end{equation*}
Since $\b>1$, both assumptions of Theorem \ref{ltgente} are met. Clearly, $\|q\|_a<\infty$, so the Jensen sum $J(H_q)$ is finite for this potential.

2. Let now $1<\a\le 2$. We claim that there is no such weight function $a$.

Assume on the contrary, that there are $a$ and $\hat a$,
which satisfy the assumptions of Theorem \ref{ltgente}, and $\|q\|_a<\infty$. Then, for $t\ge e$,
\begin{equation*}
\infty > \int_t^\infty \frac{a(x)}{x\log^\a x}\,dx\ge a(t)\int_t^\infty \frac{dx}{x\log^\a x}=\frac1{\a-1}\,\frac{a(t)}{(\log t)^{\a-1}}\,,
\end{equation*}
or
$$ a(t)\le C_1\,(\log t)^{\a-1},  \qquad t\ge e. $$
But $\a-1\le 1$, and so
$$ \int_1^\infty \frac{dt}{t a (t)} = \infty. $$
A contradiction completes the proof.
\end{example}

Part 2 of the above example by no means claims that $J(H_q)=\infty$ for those potentials.

As a final consequence of Theorem \ref{ltgente}, we study the Jensen sums for Schr\"odinger operators with compactly supported potentials.

\begin{corollary}\label{compsup}
 For any potential $q\in L^1(\br_+)$ with ${\rm supp}(q)\subset [0,R]$, $R>1$, the following inequality holds 
\begin{equation}\label{comsupbou}
J(H_q)\le 7\left[\frac1{R}+\|q\|_1\Bigl(1+\log(1+\|q\|_1)+\log R\Bigr)\right].
\end{equation}
\end{corollary}
\begin{proof}
We choose the weight functions
\begin{equation*}
a(x)=\begin{cases} 1, \ & 0<x\le R, \\
\left(\frac{\log x}{\log R}\right)^2, \ & x\ge R,
\end{cases} \qquad 
\hat a(x)=\begin{cases} x, \ & 0<x\le R, \\
x\left(\frac{\log R}{\log x}\right)^2, \ & x\ge R.
\end{cases}
\end{equation*}
Since ${\rm supp}(q)\subset [0,R]$, we have $\|q\|_a=\|q\|_1$.

Put
\begin{equation*}
\delta:=\exp\Bigl(\min\bigl(\|q\|_1R, \ \kappa\bigr)\Bigr)-1\in \Bigl(0,\frac12\Bigr], \quad \kappa=\log\frac32.
\end{equation*}
Clearly,
$$ \log(1+\delta)=\min\bigl(\|q\|_1R, \ \kappa\bigr)\le \|q\|_1R, \quad \frac{\log(1+\delta)}{\|q\|_1}\le R, $$
and so the quantity $y$ defined in \eqref{testpo} is given by
\begin{equation*}
y=\frac{\|q\|_1}{\log(1+\delta)}\,.
\end{equation*}

The right hand side of \eqref{ltgen} is the sum of two terms, $A=A_1+A_2$. The first one is
\begin{equation*}
\begin{split}
A_1 &:=y\log\frac{1+\delta}{(1-\delta)^2}=\|q\|_1+y\log\frac1{(1-\delta)^2}\leq\|q\|_1\left\{1+\frac{\log 4}{\log(1+\delta)}\right\} \\
&=\|q\|_1\left\{1+\frac{\log 4}{\min\bigl(\|q\|_1R, \ \kappa\bigr)}\right\}\,.
\end{split}
\end{equation*}
Hence,
\begin{equation*}
A_1\le
\begin{cases} \|q\|_1\left(1+\frac{\log 4}{\kappa}\right)<5\|q\|_1 , \ & \|q\|_1R\ge\kappa, \\
\|q\|_1+\frac{\log 4}{R}=\frac{\|q\|_1R+\log 4}{R}<\frac{\log 6}{R}, \ & \|q\|_1R<\kappa.
\end{cases}
\end{equation*}

To estimate the second (integral) term $A_2$, note that $y^{-1}\le R$, and so \newline $A_2=A_{21}+A_{22}$ with
\begin{equation*}
\begin{split}
A_{21} &:=\frac4{\pi}\,\|q\|_1\int_{\frac1{y}}^R \frac{dt}{t}=\frac4{\pi}\,\|q\|_1\log\frac{\|q\|_1R}{\log(1+\delta)}\,, \\
A_{22} &:=\frac4{\pi}\,\|q\|_1\log^2 R\int_R^\infty \frac{dt}{t\log^2t}=\frac4{\pi}\,\|q\|_1\log R.
\end{split}
\end{equation*}
Hence,
\begin{equation*}
A_2\le \frac4{\pi}\,\|q\|_1\log R+\frac4{\pi}\,\|q\|_1\log\frac{\|q\|_1R}{\min\bigl(\|q\|_1R, \ \kappa\bigr)}\,,
\end{equation*}
or
\begin{equation*}
A_2\le\begin{cases} \frac4{\pi}\,\|q\|_1\left(\log R+\log(\|q\|_1R)+\log\frac1{\kappa}\right), \ & \|q\|_1R\ge\kappa, \\
\frac4{\pi}\,\|q\|_1\log R, \ & \|q\|_1R<\kappa.
\end{cases}
\end{equation*}

A combination of the above bounds (with appropriate calculation of the constants) leads to \eqref{comsupbou}, as claimed.
\end{proof}
\begin{remark}
   The celebrated Blaschke condition for zeros of analytic functions on the upper half-plane reads (see \cite[Section II.2, (2.3)]{Garnett})
   \begin{equation}
     \sum_{z\in Z(f)} \frac{\im z}{1+|z|^2}<\infty.
   \end{equation}
   It holds, for instance, for functions of bounded type (ratios of bounded analytic functions).
   In view of the spectral enclosure $|z|\le\|q\|_1$, the bound $J(H_q)<\infty$ is equivalent to the Blaschke condition for zeros of the Jost function.
 \end{remark}

\section{Dissipative barrier potentials}\label{sec:db}

As in the introduction (see (\ref{eq:db-defn-intro})), let $L_{\g,R}$ denote a Schr\"odinger-Dirichlet operator on $\br_+$ with the potential
\begin{equation}
  \label{eq:db-pot-defn}
  q_{db} := i \gamma \chi_{[0,R]}, \qquad \g, R >0. 
\end{equation}
We fix $\gamma$ throughout this section and shall be interested in large $R$.
The aim of the section is  to prove the bounds for the Lieb-Thirring and Jensen sums of the eigenvalues of $L_{\g,R}$ for large enough $R$.

\subsection{Eigenvalues of Schr\"odinger operators with dissipative barrier potentials}

The value $z^2\in\s_d(L_{\g,R})$ if the equation 
\begin{equation}\label{eq1}
-y'' + i \gamma \chi_{[0,R]}(x)y =z^2 y
\end{equation}
has a solution $y\in L^2(\br_+)$  with $y(0)=0$. An integration by parts with the normalized eigenfunction gives
\begin{equation}\label{disp}
\begin{split} 
z^2 &=\int_0^\infty |y'(t)|^2dt+i\g\int_0^R |y(t)|^2dt\in\gg_+, \\
\gg_+ &:=\{\z\in\bc: \ \re\z>0, \ 0<\im\z <\g\}.
\end{split} 
\end{equation}

It shall be convenient for us to work with two different branches  $\sqpm$ of the square-root function.
$\sqpm$ have branch-cuts along $\br_\pm$, respectively, and the corresponding argument functions $\arg_\pm$ satisfy
\begin{equation*}
  \label{eq:arg-pm}
  \arg_+(\z) \in [0,2 \pi), \quad \arg_-(\z) \in [-\pi, \pi),\quad \z \in \bc; \quad \sqpm(\z)=\sqrt{|\z|}e^{\frac{i}{2}\arg_{\pm}(\z)}.
\end{equation*}

Since the solutions of the equation \eqref{eq1} are obviously computable, we may characterise the eigenvalues of $L_{\g,R}$ as the zeros 
of an explicit analytic function. Let
\begin{equation*}
  \varphi_R(z)  := \bigl(z  - \sqp(z^2 - i \gamma)\bigr) e^{iR \,\sqp(z^2 - i \gamma)} -  
  \bigl(z  + \sqp(z^2 - i \gamma)\bigr) e^{-iR \,\sqp(z^2 - i \gamma)} . 
\end{equation*}

\begin{lemma}
  \label{lem:varphi-R-eig}
  For any $R > 0$ and any $z \in \bc_+$ with $z^2 \not= i \gamma$,
  \begin{equation*}
    z^2 \in \sigma_d(L_{\g,R}) \quad \iff \quad \varphi_R(z) = 0.
  \end{equation*}
\end{lemma}
\begin{proof}
  Let $R > 0$ and $z \in \bc_+$ such that $z^2 \not= i \gamma$.
  Recall that $e_+(\cdot,z)$ denotes the Jost solution.
  Since $e_+(\cdot,z)$ spans the space of solutions of \eqref{eq1} in $L^2(\br_+)$, we have
  \begin{equation*}
    z^2 \in \sigma_d(L_{\g,R})  \quad \iff \quad e_+(0,z) = 0.
  \end{equation*}

It suffices to show that $e_+(0,z) = 0$ if and only if $\varphi_R(z) = 0$.
  Since $z \neq 0$ and $z^2 \neq i \gamma$, $e_+$ must satisfy
  \begin{equation*}
    e_+(x,z) = \begin{cases} c_1(z) e^{i x \, \sqp(z^2 - i \gamma)} + c_2(z) e^{- i x \, \sqp(z^2 - i \gamma) }, & 0 < x < R \\
      e^{i x z }, & x \geq R,
    \end{cases}
  \end{equation*}
  for some $c_j(z) \in \bc$, $j = 1,2$.
  $c_1$ and $c_2$ are determined by imposing the continuity of $e_+(\cdot,z)$ and $\frac{d}{dx} e_+(\cdot , z)$ at the point $R$,
  \begin{align*}
    c_1(z) & = \frac{\sqp(z^2 - i \gamma) + z}{2 \, \sqp(z^2 - i \gamma)} e^{ - i R \, \bigl(\sqp(z^2 - i \gamma) - z\bigr) }, \\
    c_2(z) & = \frac{\sqp(z^2 - i \gamma) - z}{2 \, \sqp(z^2 - i \gamma)} e^{ i R \, \bigl(\sqp(z^2 - i \gamma) + z\bigr) },
  \end{align*}
  and so the expression for the Jost function $e_+(0,z)$ is
  \begin{equation*}
  \begin{split}
  &{} e^{-iRz}\,e_+(0,z) \\
  &=\frac{\bigl(z+\sqp(z^2 - i \gamma)\bigr)e^{-iR\,\sqp(z^2 - i \gamma)} 
  - \bigl(z-\sqp(z^2 - i \gamma)\bigr)e^{iR\,\sqp(z^2 - i \gamma)}}{2\,\sqp(z^2 - i \gamma)} \\
  &=\cos \bigl(R\,\sqp(z^2 - i \gamma)\bigr)-izR\frac{\sin\bigl(R\,\sqp(z^2 - i \gamma\bigr))}{R\,\sqp(z^2 - i \gamma)}\,.    
  \end{split}
  \end{equation*}
  Note that it is clear from this expression that $e_+$ is an entire function.
  
  Finally, $z^2  \neq i \gamma$, so $e_+(0,z) = 0$ if and only if
  \begin{equation*}
    \varphi_R(z) = - 2 \,\sqp(z^2 - i \gamma) e^{- i R z } e_+(0, z) = 0.
  \end{equation*}
  The proof is complete.
\end{proof}

Note that, $\varphi_R(z_0)=0$ for $z_0^2=i\g$, but $z_0^2\notin\s_d(L_{\g,R})$.

\medskip

Our strategy is to derive a countable family of equations, each of which has a unique solution corresponding to exactly one zero of $\varphi_R$.
Introduce a new variable $w$ by
\begin{equation*}
  w := \sqp(z^2 - i \gamma).
\end{equation*}
For $\re z  > 0$, we have $z = \sqm(z^2)$ and so
\begin{equation}\label{eq:z-sqm-w}
  z=\sqm(w^2+i\g).
\end{equation}

Consider the family of equations
\begin{equation}
  \label{eq:fixed-point-family}
  w = G_{j,R}(w) := \frac{-B_j(w) + i A(w)}{2R}\,, \qquad j \in \bn, 
\end{equation}
where
\begin{equation*}
  A(w) := \log \left|\frac{\sqm(w^2 + i \gamma) - w}{\sqm(w^2 + i \gamma) + w}\right|
\end{equation*}
and
\begin{equation*}
  B_j(w) := \arg_- \left(\frac{\sqm(w^2 + i \gamma) - w}{\sqm(w^2 + i \gamma) + w}\right) + 2 \pi j, \qquad j\in\bn. 
\end{equation*}
Clearly,
\begin{equation}\label{boundsb}
  2\pi\Bigl(j-\frac12\Bigr) \le B_j(w)<2\pi\Bigl(j+\frac12\Bigr), \qquad j\in\bn.
\end{equation}
\begin{figure}
        \centering
        \includegraphics[width = \textwidth]{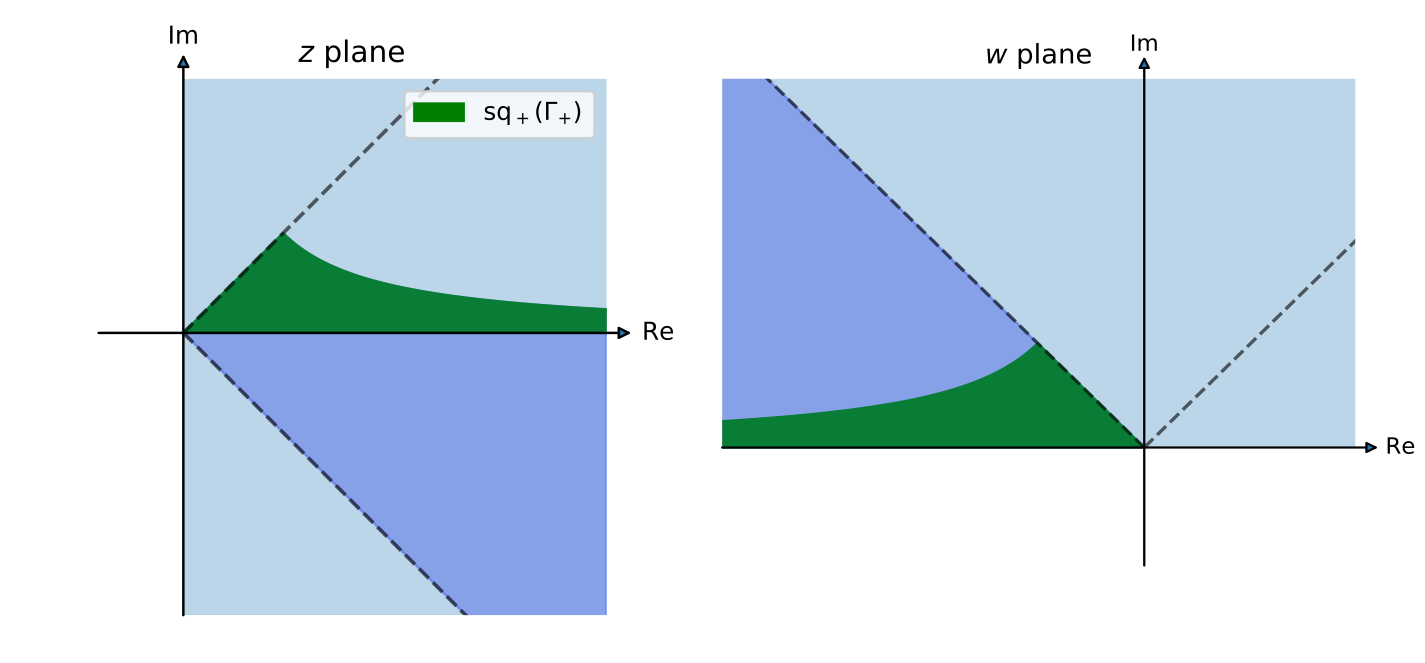}
        \caption{An illustration of the new complex variable
          $w$. Regions of identical colours are mapped to each other. }
\end{figure}

\begin{lemma}
  \label{lem:family-eigval}
  Let $R > 0$. If $w \in \bc_+$ solves  equation $\eqref{eq:fixed-point-family}$, and  $w^2 + i \gamma \in \bc_+$, 
  then $w^2 + i \gamma \in \sigma_d(L_{\g,R})$.
\end{lemma}

\begin{proof}
The equation \eqref{eq:fixed-point-family} can be written as
  \begin{equation}
    \label{eq:fixed-point-eq-logm}
    w = G_{j,R}(w) = \frac{i}{2R} \left(\log_-\left( \frac{\sqm(w^2 + i \gamma) - w}{\sqm(w^2 + i \gamma) + w}\right) + 2 \pi i j\right)
  \end{equation}
  where $\log_-$ denote the branch of the logarithm corresponding to $\arg_-$.
  Rearranging this equation, it holds that
  \begin{equation}
    \label{eq:varphi-w}
    \bigl(\sqm(w^2 + i \gamma) - w\bigr) e^{i R w} - \bigl(\sqm(w^2 + i \gamma) + w\bigr) e^{- i R w} = 0,
  \end{equation}
which is equivalent to $\varphi_R(z) = 0$, where $z$ is defined by (\ref{eq:z-sqm-w}).
  Finally, $w \neq 0$ implies $z^2 \neq i \gamma$, and
  the hypothesis $w^2 + i \gamma \in \bc_+$ ensures that $z \in \bc_+$ so, by Lemma \ref{lem:varphi-R-eig},
  we have $z^2 = w^2 + i \gamma \in \sigma_d(L_{\g,R})$.
\end{proof}

From this point on, we shall restrict our attention to solutions of \eqref{eq:fixed-point-family} in the angle
\begin{equation}\label{sinfty}
\begin{split}
F_\infty &=  \{ w \in \bc : \re w \leq 0\leq \im w, \  |\re w| \geq 2\,\im w \} \\
&=\{re^{i\theta}: \ \pi-\arctan\frac12\le\theta\le\pi, \ r \geq 0\} 
\end{split}
\end{equation}
and its subsets
$$ F_j := \{ w \in F_\infty : B_j(w) \geq 2\,|A(w)|\}, \qquad j \in \bn. $$
Since $B_{j+1}(w)=B_j(w)+2\pi$, the family $\{F_j\}_{j\ge1}$ is nested
\begin{equation*}
F_j\subset F_{j+1}, \qquad \bigcup_{j=1}^\infty F_j=F_\infty.
\end{equation*}
 As $B_j(w)\ge\pi$ for all $w \in F_\infty$, and $A(0) = 0$, the set $F_j$ is nonempty for all $j\in\bn$. 

The next result  establishes existence and uniqueness of solutions in the regions $F_j$ for each equation \eqref{eq:fixed-point-family} and
large enough $R$. Precisely, we assume that
\begin{equation}\label{bigr}
R\geq C_0\Bigl(\gamma^{3/4}+\gamma^{-3/4}\Bigr), \qquad C_0= 600.
\end{equation}

\begin{proposition}
  \label{prop:ex-and-uniq}
  For all $R$ satisfying $\eqref{bigr}$ and all $j \in \bn$, the equation $\eqref{eq:fixed-point-family}$
  has a unique solution in $F_\infty$ which lies in $F_j$. For different equations the solutions are different.
\end{proposition}
\begin{proof}
A key ingredient of the proof is the contraction mapping principle (see, e.g., \cite[Theorem V.18]{ReedSim1}) on the 
complete metric space $(F_j,|\cdot|)$ with the usual absolute value on $\bc$ as a distance.

Fix $j \in \bn$.
Suppose we can show that for $R$ satisfying \eqref{bigr},
  \begin{enumerate}
  \item[(a)] $G_{j,R}: F_j \to F_j$,
  \item[(b)] $G_{j,R}$ is a strict contraction mapping.
  \end{enumerate}

Then, the map $G_{j,R}:F_j \to F_j$ has a unique fixed point, and so the equation $w = G_{j,R}(w)$ has a unique solution in $F_j$.
Moreover, there are no solutions for the latter equation outside $F_j$. Indeed, any solution $w\in F_\infty$ satisfies
$$ w=G_{j,R}(w)=\frac{-B_j(w) + i A(w)}{2R} \ \Rightarrow \ B_j(w)\ge 2|A(w)| $$
so $w \in F_j$. 
So, it suffices to prove the statements (a) and (b) above.

Put
$$ w=u+iv, \quad z=\sqm\bra{w^2+i\g}=x+iy. $$
Let us show first that for each $w\in F_\infty$,
  \begin{equation}
    \label{eq:Sinf-facts}
x=\re\sqm\bra{w^2 + i \gamma} \geq 0, \quad  |y|=\bigl|\im \sqm\bra{w^2 + i \gamma}\bigr| \leq x=\re\sqm\bra{w^2 + i \gamma}.
  \end{equation}
Indeed, the first inequality follows from the definition of $\sqm$. As for the second one, since $\re(z^2)=\re(w^2)$ and $|u|\ge 2v$, we have
$$ u^2-v^2=x^2-y^2, \quad x^2=y^2+u^2-v^2\ge y^2+3v^2 \ \Rightarrow \ |y|\le x, $$
as claimed.

Step 1. To prove the statement (a), we show first that the following inequalities hold 
  \begin{enumerate}
  \item $\re G_{j,R}(w)<0\le\im G_{j,R}(w)$, \ $w\in F_\infty$,
  \item $|\re G_{j,R}(w)|\ge2\,\im G_{j,R}(w)$, \ \ \ $w\in F_j$.
  \end{enumerate}
In view of the definition of $B_j(w)=-2R\,\re G_{j,R}(w)$, and the bounds \eqref{boundsb} for $B_j$, the left inequality in (1) is obvious. 
To prove the right one, it suffices to show that $A(w) \geq 0$ for all $w \in F_\infty$.
We write
\begin{equation*}
|z\pm w|^2 =|z|^2+|w|^2\pm 2\,\re(\bar w z) =|z|^2+|w|^2\pm 2(ux+vy),
\end{equation*}
and so
$$ |z-w|^2-|z+w|^2=-4(ux+vy). $$
As we know, $|u|\ge 2v$ for $w\in F_\infty$, and also $x\ge |y|$, by \eqref{eq:Sinf-facts}. Hence,
$$ |vy|\le \frac{|u|x}2\le |ux|, \quad ux+vy\le ux+|vy|\le ux+|ux|=0, $$
which implies 
$$ |z-w|^2-|z+w|^2=-4(ux+vy)\ge0, \quad A(w)=\log\Bigl|\frac{z-w}{z+w}\Bigr|\ge0, $$
and (1) follows.
(2) is just the definition of $F_j$. So, $G_{j,R}: F_j \to F_\infty$.

Next, we want to check that for $R$ satisfying \eqref{bigr},
\begin{equation}\label{intosj}
B_j\bigl(G_{j,R}(w)\bigr)\ge 2\,|A\bigl(G_{j,R}(w)\bigr)|, \qquad w\in F_j,
\end{equation}
or, in other words, $G_{j,R}(w)\in F_j$.
It is shown above that, for $w \in F_j$, we have  $G_{j,R}(w)\in F_\infty$ and $|A\bigl(G_{j,R}(w)\bigr)|=A\bigl(G_{j,R}(w)\bigr)\ge0$. Then,
\begin{equation*}
\begin{split}
A\bigl(G_{j,R}(w)\bigr) &=\log\frac{\Bigl|\sqm\bra{G_{j,R}^2(w)+i\g}-G_{j,R}(w)\Bigr|^2}{\g} \\
&\le\log\frac{2(4|G_{j,R}(w)|^2+\g)}{\g}=\log\left(\frac{8|G_{j,R}(w)|^2}{\g}+2\right).
\end{split}
\end{equation*}
For $w\in F_j$ one has $2|A(w)|\le B_j(w)$, and so, by \eqref{boundsb},
$$ |G_{j,R}(w)|^2=\frac{A^2(w)+B_j^2(w)}{4R^2}\le \frac{5B_j^2(w)}{16 R^2}\le \frac{5\pi^2}{4R^2}\Bigl(j+\frac12\Bigr)^2.$$
Hence,
\begin{equation}\label{eq:A-upper}
  A\bigl(G_{j,R}(w)\bigr)\le \log\Bigl(2+\frac{10\pi^2(j+\frac12)^2}{\g R^2}\Bigr).
\end{equation}
Clearly, $10\pi^2<\gamma R^2$ for $R$ satisfying \eqref{bigr}, so we come to
\begin{equation}\label{bounda} 
A\bigl(G_{j,R}(w)\bigr)\le \log\Bigl(2+\Bigl(j+\frac12\Bigr)^2\Bigr)<\log\Bigl(2\Bigl(j+\frac12\Bigr)^2\Bigr), \quad j\in\bn. 
\end{equation}
Elementary calculus shows that
$$ \log 2+2\log\Bigl(j+\frac12\Bigr)<\pi\Bigl(j-\frac12\Bigr), \quad j\in\bn, $$
and so $2\,A\bigl(G_{j,R}(w)\bigr)\le B_j\bigl(G_{j,R}(w)\bigr)$, which completes the proof of \eqref{intosj}. The statement (a) is verified.

Step 2. We shall proceed with the statement (b). Let $h$ denote the function
  \begin{equation}
    \label{eq:h-defn}
    h(w) := \frac{\sqm\bra{w^2 + i \gamma} - w}{\sqm\bra{w^2 + i \gamma} + w} = \frac{1}{i \gamma} \bigl(\sqm\bra{w^2 + i \gamma} - w\bigr)^2.
  \end{equation}
In view of \eqref{eq:Sinf-facts} and $u=\re w\le0$, it is easy to see that for each $w\in F_\infty$, 
$$ \sqm\bra{w^2 + i \gamma}-w\in G:=\{\z\in\bc: \ \re \z \geq 0, \ |\im\z|\le\re\z \},  $$
and so $h: F_\infty \to \ovl\bc_-$.  

We conclude that the branch $\log_-$  of the logarithm (corresponding to $\arg_-$) is continuously differentiable on $h(F_\infty)$.
By the expression for $G_{j,R}$ in  (\ref{eq:fixed-point-eq-logm}), $G_{j,R}$ is continuously differentiable on $F_\infty$.
A direct computation yields
   \begin{equation*}
     \frac{d}{d w}G_{j,R}(w) = \frac{-i}{R \, \sqm\bra{w^2 + i \gamma}}.
   \end{equation*}
It is easy to show (see the definition of $F_\infty$ \eqref{sinfty}) that
$$ \min_{w\in F_\infty}|w^2+i\gamma|=C\gamma, \qquad C=\cos\Bigl(2\arctan\frac12\Bigr)>\frac1{2}, $$
and so
$$ \left|\frac{d}{dw}G_{j,R}(w)\right|<1, \qquad w\in F_\infty, $$
as long as $R$ satisfies \eqref{bigr}. Hence, $G_{j,R}:F_j \to F_j$ is a strict contraction mapping for such $R$,
completing the proof.
\end{proof}

\subsection{The number of eigenvalues and Lieb--Thirring sums for $L_{\g,R}$}

Now that existence of solutions for the family of equations (\ref{eq:fixed-point-family}) has been established, we may prove lower bounds for 
Lieb--Thirring sums. Throughout the remainder of the section, we assume that $j \in \bn$ and $R$ satisfies \eqref{bigr},
and we let $w_j = w_j(\gamma,R) \in F_j$ denote the unique solution to the equation $w = G_{j,R}(w)$ in $F_j$.

As it turns out, one has to impose some restriction on the values $j$ to guarantee that $w_j$ corresponds to an eigenvalue. Precisely,
assume that
\begin{equation}\label{restj}
1\le j\le M_R:=\left\lfloor\frac1{32 \pi} \frac{\gamma R^2}{\log R}\right\rfloor\,.
\end{equation}

\begin{lemma}\label{lem:JR-cond}
For $R$ satisfying $\eqref{bigr}$ and $j$ satisfying $\eqref{restj}$, the inequalities 
\begin{equation}\label{imwj}
-\frac{\g}2\le \im w_j^2\le0
\end{equation}
hold, so $z_j^2=w_j^2+i\g\in\bc_+$ and $z_j^2\in\s_d(L_{\g,R})$.
\end{lemma}
\begin{proof}
  
  Firstly, we claim that for all $\g>0$ and $R$ satisfying \eqref{bigr}, we have
  \begin{equation}
    \label{eq:Phi-gam-lower}
  \Phi_\g(R):=\frac{\gamma R^2}{\log R} >\frac{C_0^2}{2\log C_0}.
  \end{equation}
Since $R\ge 2C_0>\sqrt{e}$, the function $\Phi_\gamma(R)$ is monotonically increasing and for each $\g>0$
\begin{equation*}
\begin{split}
\Phi_\g(R)  &\ge f(\g):=C_0^2\,\frac{\g\Bigl(\g^{3/4}+\g^{-3/4}\Bigr)^2}{\log C_0 +\log\Bigl(\g^{3/4}+\g^{-3/4}\Bigr)} \\
&=C_0^2\,\frac{\g^3+2\g^{3/2}+1}{\sqrt{\g}\log C_0 +\sqrt{\g}\log\Bigl(\g^{3/2}+1\Bigr)-\frac34\sqrt{\g}\log\g}\,. 
\end{split}
\end{equation*}
Since $f(\g)\le f(\g^{-1})$, $0<\g\le1$, and $C_0>e^2$, we see that
\begin{equation*}
\min_{\g>0}f(\g) =\min_{0<\g\le 1}f(\g)\ge \frac{C_0^2}{\log C_0+\log 2+\frac3{2e}}>\frac{C_0^2}{\log C_0+2}>\frac{C_0^2}{2\log C_0}\,, 
\end{equation*}
proving \eqref{eq:Phi-gam-lower}.

Next, we have
\begin{equation}\label{mrbelow}
M_R> \frac{\Phi_\g(R)}{32\pi}-1>\frac1{32\pi}\,\frac{C_0^2}{2\log C_0}-1\ge 1,
\end{equation}
 as long as
$$ \frac{C_0^2}{2\log C_0}>64\pi, $$
which certainly true for the value $C_0$ in \eqref{bigr}. By \eqref{mrbelow},
\begin{equation*}
\frac{\Phi_\g(R)}{32\pi}>2, \qquad \frac{\Phi_\g(R)}{96\pi}>\frac23>\frac12\,.
\end{equation*}
We assume that $1\le j\le M_R$, so
\begin{equation}\label{jplus12}
j+\frac12\leq \frac{\Phi_\g(R)}{32\pi}+\frac12<\frac{\Phi_\g(R)}{24\pi}=\frac1{24\pi}\,\frac{\g R^2}{\log R}\,. 
\end{equation}
Since $A(w_j)\ge0$, $B_j(w_j)>0$, we have
\begin{equation}\label{wjsq}
\begin{split}
w_j^2 &=G_{j,R}^2(w_j)=\frac{B_j^2(w_j)-A^2(w_j)-2i B_j(w_j)A(w_j)}{4R^2}\,, \\
\im w_j^2 &=-\frac{B_j(w_j)A(w_j)}{2R^2}\le0.
\end{split}
\end{equation}
To prove the lower bound in \eqref{imwj}, we apply \eqref{bounda} and \eqref{jplus12}
\begin{equation*}
A(w_j)\le  \log 2+2\log\Bigl(j+\frac12\Bigr) < 2\log\g+4\log R,
\end{equation*}
and hence
\begin{equation*}
B_j(w_j)A(w_j)\le 4\pi\Bigl(j+\frac12\Bigr)(\log\g+2\log R).
\end{equation*}
But $R>\g^{3/4}$, $\log R>\frac34\log \g$, and so, by \eqref{jplus12},
\begin{equation*}
\log\g+2\log R <\frac{10}3\log R, \ \
B_j(w_j)A(w_j) \le \frac16\,\frac{\g R^2}{\log R}\cdot\frac{10}3\log R<\g R^2.
\end{equation*}
The lower bound in \eqref{imwj} follows. The remaining claims follow from an application of Lemma \ref{lem:family-eigval}.
The proof is complete.
\end{proof}

The result of Lemma \ref{lem:JR-cond} immediately implies a lower bound for the  number $N(L_{\g,R})$ of eigenvalues of $L_{\g,R}$, 
counting algebraic multiplicities. 

\begin{corollary}
  \label{col:number}
  For $R$ satisfying $\eqref{bigr}$, we have the lower bound
 \begin{equation*}
    N(L_{\g,R}) \geq \left\lfloor\frac1{32 \pi} \frac{\gamma R^2}{\log R}\right\rfloor.
  \end{equation*}
\end{corollary}

The next result amplifies the above corollary and will be used in our study of the sums $S_{\alpha, \beta}(H_q)$ below.  
An analogous result for Schr\"odinger operators on the real line has previously been obtained by Cuenin in \cite[Theorem~4]{CuenSparse},
by a different method.
Let $N(L_{\gamma,R};\Omega)$ denote the number of eigenvalues of $L_{\gamma, R}$  in a given region $\Omega \subset \bc$, counting algebraic multiplicities.

\begin{proposition}\label{prop:number-box}
   There exists constants  $R_0, C_1> 0$, depending only on $\gamma$, such that for the regions
  \begin{equation}
    \label{eq:Sigma_R-defn}
    \Sigma_{R} := \left\lbrace\lambda \in \bc: \tfrac{\gamma}{2} \leq \im(\lambda) \leq \gamma,
       \frac{C^{-1}_1R^2}{\log^2R} \leq |\lambda| \leq \frac{ C_1R^2}{\log^2R} \right\rbrace
  \end{equation}
   and all $R \geq R_0$, we have
  \begin{equation}
    \label{eq:lower-box}
    N(L_{\gamma,R};\Sigma_R) \geq \frac{1}{128 \pi} \frac{\gamma R^2}{\log R}. 
  \end{equation}
\end{proposition}
\begin{proof}
  In this proof, we shall say that a statement holds for large enough $R$ if there exists $R_0 = R_0(\gamma) > 0$
  such that the statement holds for all $R \geq R_0$.
  Furthermore, $C = C(\gamma) > 0$ will denote a constant that may change from line~to~line.

  Consider the unique solution $w_j = w_j(\gamma, R)$ of the equation $w = G_{j,R}(w)$ in $F_j$, which exists for large enough $R$, with
  \begin{equation}\label{eq:j-range-box}
    \left\lceil \frac{1}{64 \pi} \frac{\gamma R^2}{\log R} \right\rceil \leq j \leq \left\lfloor \frac{1}{32\pi} \frac{\gamma R^2}{\log R} \right\rfloor.
  \end{equation}
  By Lemma \ref{lem:JR-cond}, $\lambda_j := w_j^2 + i \gamma$ is an eigenvalue of $L_{\gamma,R}$ with $\tfrac{\gamma}{2} \leq \im(\lambda_j) \leq \gamma$.

  By (\ref{eq:A-upper}), we have
  \begin{equation*}
    |A(w_j)| = |A(G_{j,R}(w_j))| \leq \log \left( 2 + \frac{10 \pi^2 (j + \tfrac{1}{2})^2}{\gamma R^2} \right) \leq C R
  \end{equation*}
  for large enough $R$.
  Using the inequality $B_j(w_j) \geq 2 \pi(j - \tfrac{1}{2})$ and the lower bound in (\ref{eq:j-range-box}), we have
  \begin{equation*}
    |\lambda_j| \geq |w_j|^2 - \gamma = \frac{|B_j(w_j)|^2 + |A(w_j)|^2}{4R^2} - \gamma \geq \frac{CR^2}{\log^2R}   
  \end{equation*}
  for large enough $R$.
  On the other hand,  using the inequality $B_j(w_j) \leq 2 \pi(j + \tfrac{1}{2})$ and the upper bound in (\ref{eq:j-range-box}), we have
  \begin{equation*}
    |\lambda_j| \leq  \frac{|B_j(w_j)|^2 + |A(w_j)|^2}{4R^2} + \gamma \leq \frac{CR^2}{\log^2R} 
  \end{equation*}
  for large enough $R$.
  It follows that $\lambda_j \in \Sigma_R$ for some constant $C_1 = C_1(\gamma) > 0$ and all large enough $R$.
  Finally, we have
  \begin{equation*}
    N(L_{\gamma,R};\Sigma_R) \geq  \left\lfloor \frac{1}{32\pi}\frac{\gamma R^2}{\log R}\right\rfloor
    -  \left\lceil \frac{1}{64 \pi} \frac{\gamma R^2}{\log R}  \right\rceil \geq \frac{1}{128 \pi}\frac{\gamma R^2}{\log R}
  \end{equation*}
  for large enough $R$, completing the proof.
\end{proof}

\begin{remark}
  An upper bound for the number of eigenvalues for Schr\"odinger operators with potentials of the form $q_R = q+i\gamma\chi_{[0,R]}$,
  where $q$ is compactly supported,
  is obtained in \cite[Theorem 8]{StepaBounds}
\begin{equation}\label{upnum}
N(H_{q_R})\le \frac{11}{\log 2}\,\frac{\gamma R^2}{\log R}
\end{equation}
for large enough $R$. Our particular case corresponds to $q\equiv 0$ and demonstrates that \eqref{upnum} is optimal.
\end{remark}

\smallskip

The result of Theorem \ref{quanbou} states that for each $\ep>0$ there exists a constant 
$K(\ep)>0$, independent from $q$, so that 
$$ S_\ep(H_q)\le K(\ep)\,\|q\|_{1}^{1+\ep} $$ 
for any integrable potential $q$. 
Our goal here is to obtain corresponding lower bounds for the operators $L_{\g,R}$ with potentials $q_{db}$ \eqref{eq:db-pot-defn} and, thereby, to demonstrate the optimal 
character of this upper bound with respect to $\varepsilon$.
Precisely, we will show that the value $S_0(L_{\g,R})$ tends to infinity fast enough as $R\to\infty$. 

\begin{theorem}[= Theorem \ref{th:db-lower-intro}]\label{litrbelow}
  Suppose that $R$ satisfies $\eqref{bigr}$.
  
$(i)$ We have the lower bound
\begin{equation}\label{belowlts0}
S_0(L_{\g,R})\ge \frac{\|q_{db}\|_{1}}{16\pi}\,\log R=\frac{\g R}{16\pi}\,\log R.
\end{equation}

$(ii)$ Let $0<\ep<1$. Under the stronger assumption on $R$
\begin{equation}\label{ranger1}
R\ge \frac4{e^2\g}(64\pi)^{2/\ep}+1,
\end{equation}
we have the lower bound
\begin{equation}\label{ep>0}
S_\ep(L_{\g,R})\ge \frac1{256\pi\ep}\,\frac{(\g R)^{1+\ep}}{\log^\ep R}\,.
\end{equation}
\end{theorem}
\begin{proof}
(i). The bound from below for $S_0(L_{\g,R})$ arises when we take a subset of the eigenvalues, precisely, $\l_j=z_j^2=w_j^2+i\g$, with 
$j$ from \eqref{restj}. So, for $\ep=0$ we have, in view of Lemma \ref{lem:JR-cond},
\begin{equation}
  \label{eq:mod-for-remark-1}
S_0(L_{\g,R})\ge \sum_{j=1}^{M_R} \frac{\im(w_j^2+i\g)}{|w_j^2+i\g|^{1/2}}\ge 
\frac{\g}{2}\sum_{j=1}^{M_R} \frac1{{\sqrt\g}+|w_j|}\,.
\end{equation}
But, owing to \eqref{boundsb},
\begin{equation*}
|w_j|^2=|G_{j,R}(w_j)|^2=\frac{|A(w_j)|^2+|B_j(w_j)|^2}{4R^2}\le \frac{5|B_j(w_j)|^2}{16R^2}\le \frac{5\pi^2}{4R^2}\Bigl(j+\frac12\Bigr)^2, 
\end{equation*}
and so
\begin{equation*}
S_0(L_{\g,R})\ge \frac{\g}{2}\sum_{j=1}^{M_R} \frac1{\sqrt\g+\frac{2\pi}{R}(j+1)}\,.
\end{equation*}

An elementary inequality
$$ \sum_{j=1}^N \frac1{a+b(j+1)}>\int_2^{N+1} \frac{dx}{a+bx}=\frac1{b}\log\frac{a+b(N+1)}{a+2b} $$
with $a=\sqrt\g$, $b=2\pi R^{-1}$, $N=M_R$, gives
\begin{equation}\label{ltbel1}
S_0(L_{\g,R})\ge \frac{\g R}{4\pi}\log\frac{\sqrt\g+2\pi R^{-1}(M_R+1)}{\sqrt\g+4\pi R^{-1}}\ge\frac{\g R}{4\pi}
\log\frac{1+\frac{\sqrt\g R}{16\log R}}{1+\frac{4\pi}{\sqrt\g R}}\,.
\end{equation}

Let us check that for $R$ satisfying \eqref{bigr}, one has
\begin{equation*}
\frac{1+\frac{\sqrt\g R}{16\log R}}{1+\frac{4\pi}{\sqrt\g R}}>R^{1/4}, \ \ \frac{\sqrt\g R}{16\log R}+1>R^{1/4}+\frac{4\pi}{\sqrt\g R^{3/4}}.
\end{equation*}
Indeed,
\begin{equation*}
\frac{\sqrt\g R^{3/4}}{16\log R}=\frac{\g^{1/2}R^{2/3}}{16}\,\frac{R^{1/12}}{\log R}>\frac{C_0^{2/3}}{16}\,\frac{e}{12}>1
\end{equation*}
as long as 
$$ C_0>\left(\frac{192}{e}\right)^{3/2}, $$
which is true for $C_0$ in \eqref{bigr} (at this point the value $C_0= 600$ comes about). Next,
$$ \frac{4\pi}{\sqrt{\g}R^{3/4}}=\frac{4\pi}{\g^{1/2}R^{2/3}R^{1/12}}<\frac{4\pi}{C_0^{2/3}}<1 $$
as long as $C_0>(4\pi)^{3/2}$.
The bound \eqref{belowlts0} follows directly from \eqref{ltbel1}.

(ii). We have, as above in (i),
\begin{equation}
  \label{eq:mod-for-remark-2}
S_\ep(L_{\g,R})\ge \frac{\g}2\sum_{j=1}^{M_R} \frac1{\g^{\frac{1-\ep}2}+|w_j|^{1-\ep}}
\ge \frac{\g^{\frac{1+\ep}2}}2 \sum_{j=1}^{M_R}\frac1{1+\Bigl(\frac{2\pi}{\sqrt\g R}(j+1)\Bigr)^{1-\ep}},
\end{equation}
and so
$$ S_\ep(L_{\g,R})\ge \frac{\g^{1+\frac{\ep}2}R}{4\pi}\int_{\b_1}^{\b_2}\frac{dy}{1+y^{1-\ep}}\,, \quad \b_1:=\frac{4\pi}{\sqrt\g R}, \quad
\b_2:=\frac{2\pi(M_R+1)}{\sqrt\g R}\,. $$
An elementary inequality $1+y^{1-\ep}\le2^\ep(1+y)^{1-\ep}$ leads to the bound
\begin{equation*} 
\begin{split}
S_\ep(L_{\g,R}) &\ge \frac{\g^{1+\frac{\ep}2}R}{4\pi 2^\ep}\int_{\b_1+1}^{\b_2+1}\frac{dt}{t^{1-\ep}}
=\frac{\g^{1+\frac{\ep}2}R}{4\pi\ep 2^\ep}\Bigl\{(1+\b_2)^\ep-(1+\b_1)^\ep\Bigr\} \\
&=I_1-I_2.
\end{split}
\end{equation*}

We apply once again $(1+\b_2)^\ep\ge 2^{\ep-1}(1+\b_2^\ep)$ to estimate the first term
\begin{equation}\label{boundi1}
I_1 \ge \frac{\g^{1+\frac{\ep}2}R}{8\pi\ep}\bigl(1+\b_2^\ep)\ge \frac{\g^{1+\frac{\ep}2}R}{8\pi\ep}\left(\frac{2\pi(M_R+1)}{\sqrt\g R}\right)^\ep 
> \frac{(\g R)^{1+\ep}}{128\pi\ep}\frac1{\log^\ep R}\,.
\end{equation}
Concerning the second term, note that \eqref{ranger1} implies $\sqrt{\g} R>8$, and so
$$ (1+\b_1)^\ep=\Bigl(1+\frac{4\pi}{\sqrt\g R}\Bigr)^\ep<1+\frac{\pi}2<\pi. $$
Then
\begin{equation*}
I_2\le \frac{\g^{1+\frac{\ep}2}R}{4\ep 2^\ep}=\frac{(\g R)^{1+\ep}}{4\ep 2^\ep}\frac1{(\sqrt\g R)^\ep}<\frac{(\g R)^{1+\ep}}{4\ep}\,\frac1{(\sqrt\g R)^\ep}\,. 
\end{equation*}
But, under assumption \eqref{ranger1},
$$ \frac{\sqrt{\g} R}{\log R}=\frac{\sqrt{\g} R^{1/2}}{\log R}\,R^{1/2}\ge\frac{\sqrt\g e}2\,R^{1/2}\ge (64\pi)^{1/\ep}\,, $$
so
$$ \left(\frac{\log R}{\sqrt\g R}\right)^\ep\leq\frac1{64\pi}\,, \qquad \frac1{(\sqrt\g R)^\ep}\leq\frac1{64\pi\log^\ep R}. $$
Hence,
$$ I_2\le \frac{(\g R)^{1+\ep}}{256\pi\ep\log^\ep R}\,. $$
Comparing the latter with \eqref{boundi1}, we come to \eqref{ep>0}. The proof is complete.
\end{proof}

\begin{remark}\label{rem:result-S-B}
  The same methods lead to lower bounds for more general sums, which were considered in \cite{BogStam}.
  Let $p \geq 1$. A slight modification of the proof of Theorem \ref{litrbelow} (i) yields
  \begin{equation}
    \label{eq:bog-stam-p-half}
   \sum_{\lambda \in \sigma_d(L_{\g,R})} \frac{\mathrm{dist}^p(\lambda,\br_+)}{|\lambda|^{1/2}} \geq \frac{\gamma^p R \log R}{8\pi\cdot 2^p}\,,
  \end{equation}
  provided $R$ satisfies \eqref{bigr}.
  Indeed, the only place in the proof of Theorem \ref{litrbelow} (i) that needs to be modified is \eqref{eq:mod-for-remark-1}, and there
  we use the inequality 
  $$ \im(w_j^2 + i \gamma)^p \geq (\gamma/2)^p. $$
  
  Furthermore, by the spectral enclosure \cite{FLS} mentioned in the introduction, we have 
  $$ |\lambda|^{s - 1/2} \leq \|q_{db}\|_1^{2s -1} = (\gamma R)^{2s -1},\quad \lambda \in \sigma_d(L_{\g,R}), \quad s\ge\frac12\,, $$ 
  so it follows from \eqref{eq:bog-stam-p-half} that
  \begin{equation}
    \label{eq:bog-stam-p-s}
   \sum_{\lambda \in \sigma_d(L_{\g,R})} \frac{\mathrm{dist}^p(\lambda,\br_+)}{|\lambda|^s} \geq \frac{1}{8\pi\cdot 2^p}\frac{\gamma^p R \log R}{(\gamma R)^{2s -1}}\,.
 \end{equation}
 
  Now take $R = n$ and $\gamma = n^{-1}$ for $n \in \bn$. Then, $R$ satisfies \eqref{bigr}, and so \eqref{eq:bog-stam-p-s} holds, for large enough $n$.  
  Noting that $\|q_{db}\|^p_{L^p(\br_+)} = \gamma^p R$ and $\gamma R = 1, $ and taking the limit $n \to \infty$, we conclude that
  \begin{equation}
    \label{eq:sup-Hq-infty}
    \sup_{0\neq q \in L^p(\br_+)\cap L^1(\br_+)} \frac{1}{\|q\|^p_{L^p(\br_+)}} 
    \sum_{\lambda \in \sigma_d(H_q)} \frac{\mathrm{dist}^p(\lambda,\br_+)}{|\lambda|^s} = \infty.
  \end{equation}
   In view of Proposition \ref{evenext} below, the statement \eqref{eq:bog-stam-p-s} holds analogously for Schr\"odinger operators
  on $L^2(\br)$ with symmetric potentials $i \g \chi_{[-R,R]}$, hence \eqref{eq:sup-Hq-infty} holds for Schr\"odinger operators
   on $L^2(\br)$, which implies \cite[Theorem 9]{BogStam}.   
\end{remark}

Recall that the generalised Lieb--Thirring sum $S_{\a,\b}(H_q)$ is defined by \eqref{genltsum}. The problem we are interested in now 
is the range of positive parameters  $(\a,\b)$ for which
\begin{equation*}
\css_{\a,\b}:=\sup_{0\neq q \in L^1(\br_+)}\frac{S_{\a,\b}(H_q)}{\|q\|_1}<\infty.
\end{equation*}
The results are illustrated in Figure \ref{fig:S_ab}.
\begin{proposition}\label{generltsup}
We have
\begin{equation}\label{genltfin}
\css_{\a,\b}<\infty, \qquad {\rm for} \ \ \a>\frac12\,, \ \b\ge1,
\end{equation}
and
\begin{equation}\label{genltinf}
\css_{\a,\b}=\infty, \qquad {\rm for} \ \ \a > 0\,, \ 0<\b<1, \ \ {\rm and} \ \ 0<\a\le\frac12\,, \ \b=1. 
\end{equation}
\end{proposition}
\begin{proof}
  Theorem \ref{quanbou} implies that we have $\css_{\alpha,1} < \infty$ for $\alpha > \tfrac{1}{2}$.
  Furthermore, by $\mathrm{dist}(\lambda, \br_+) \leq |\lambda|$, the function $f(\beta) = S_{\a,\b}(H_q)$ is monotone decreasing for fixed $\alpha$,
  from which (\ref{genltfin}) follows.
  
  By Proposition \ref{prop:number-box}, for $\a > 0$ and $0<\b < 1$, we have
 \begin{align*}
    S_{\a,\b}^{2 \a}(H_q) & \geq N(L_{\g,R};\Sigma_R) \inf_{\l \in \Sigma_R} \left( \frac{\mathrm{dist}(\l,\br_+)}{|\l|} \right)^\b |\l|^\a \\
                        & \geq \frac{1}{128 \pi} \frac{\g R^2}{\log R} \left( \frac{\g}{2} \right)^\b \left( \frac{\min \left\lbrace C_1,C^{-1}_1 \right\rbrace R^2}{\log^2 R} \right)^{\a-\b} \\
                        & = C \frac{R^{2(1-\b)}}{(\log R)^{1 + 2 \a - 2\b}} (\g R)^{2 \a}
 \end{align*}
 for some constant $C = C(\g) > 0$ and all large enough $R$.
 The first statement in (\ref{genltinf})  follows by considering the limit $R \to \infty$. 

  By (\ref{eq:bog-stam-p-s}) with $p =\beta = 1$ and $s = 1 - \alpha \geq \tfrac{1}{2}$, we have
\begin{equation*}
  S_{\alpha, \beta}^{2 \a}(L_{\gamma,R}) = \sum_{\lambda \in \sigma_d(L_{\g,R})}  \frac{\mathrm{dist}(\l, \br_+)}{|\lambda|^{1-\a}}
  \geq \frac{1}{16 \pi} (\gamma R)^{2 \alpha} \log R  
\end{equation*}
for large enough $R$.
The second statement in (\ref{genltinf})  follows by again considering the limit $R \to \infty$. 
\end{proof}
\begin{figure}
        \centering
        \includegraphics[width = 0.75\textwidth]{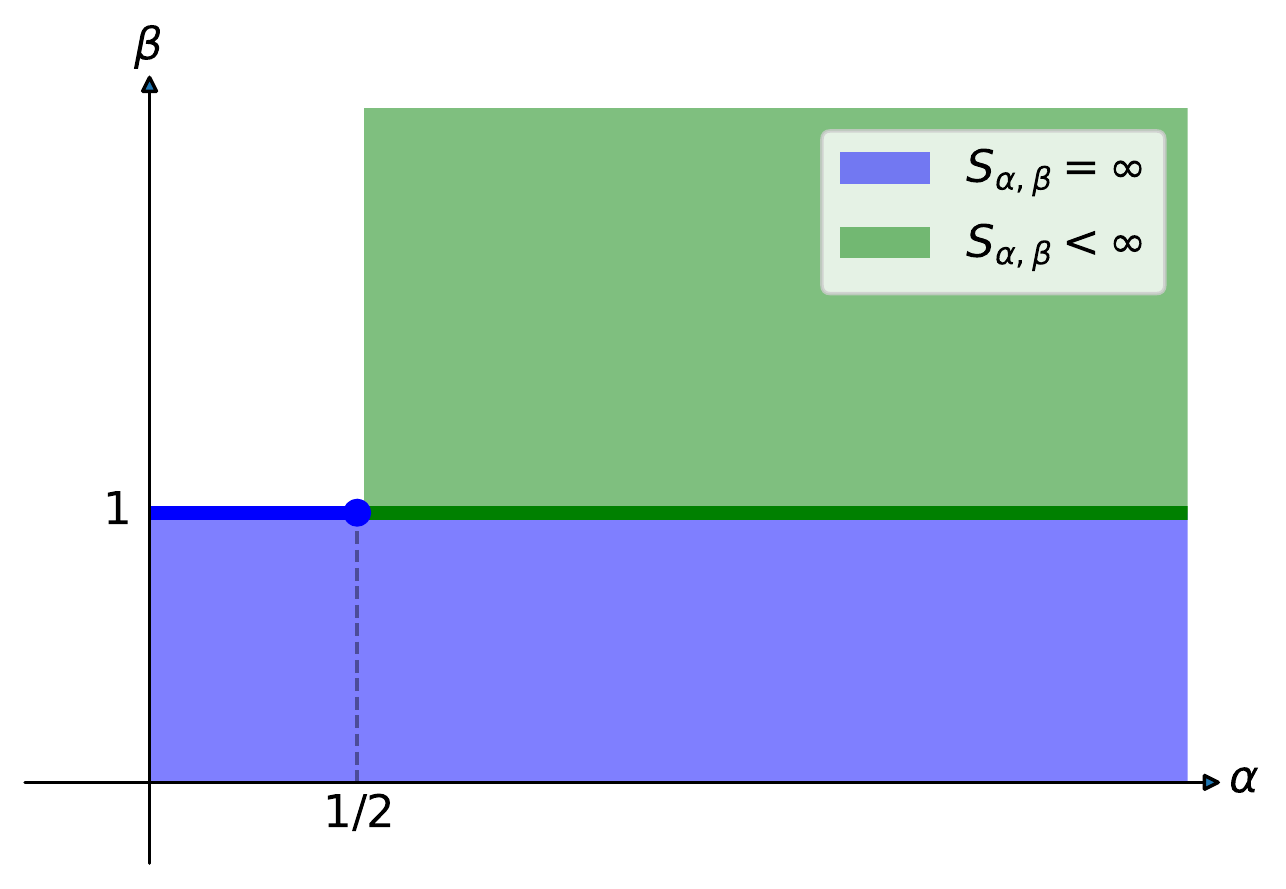}
        \caption{An illustration of Proposition \ref{generltsup}.}\label{fig:S_ab}
\end{figure}

We are in a position now to obtain a two-sided bound for the Jensen sums $J(L_{\g,R})$.  Recall that $\|q_{db}\|_1=\gamma R$.

\begin{proposition}
  \label{prop:two-sided}
For all $R$ satisfying \eqref{bigr}, the following two-sided inequality holds
\begin{equation}\label{twojen}
\frac1{32\pi}\le \frac{J(L_{\g,R})}{\gamma R\log R}\le 42.
\end{equation}
\end{proposition}
\begin{proof}
The lower bound is a direct consequence of \eqref{belowlts0} and \eqref{jenvslt}.
To prove the upper bound, we apply Corollary \ref{compsup}, so
\begin{equation*}
J(L_{\g,R})\le 7\Bigl[\frac1{R}+\g R+\g R\log R+\g R\log(1+\g R)\Bigr].
\end{equation*}
Note that \eqref{bigr} implies $R>e$ and $R^2>\gamma+\gamma^{-1}+1$. Hence,
$$ \frac1{R}<\gamma R\log R, \quad \gamma R<\gamma R\log R, \quad \log(1+\gamma R)<3\log R, $$
and inequality \eqref{twojen} follows.
\end{proof}

\section{An integrable potential with divergent Jensen sum}\label{sec:q-inf}

The aim of this section is to construct a potential $q_\infty \in L^1(\br_+)$ such that $J(H_{q_\infty}) = \infty$.
We shall begin, in Sections \ref{subsec:half} and \ref{subsec:full}, by collecting some well-known facts about Schr\"odinger operators on both the half-line
and the full real line.
We shall then proceed to prove two spectral approximation lemmas in Section \ref{subsec:approx}.
These will give us information on the eigenvalues of Schr\"odinger operators on the half-line, for potentials consisting of a sum of  compactly supported
functions whose supports are separated  from one another by large enough distances. The consideration of Schr\"odinger operators on the full real line
is required in order to formulate one of these lemmas. With these tools at hand, the potential $q_\infty$ is constructed in Section \ref{subsec:main}.
\subsection{Case of the half-line}\label{subsec:half}
Consider the following differential equation on the positive half-line $\br_+$
\begin{equation}\label{difeqr+}
h[y]:=-y''+q(x)y=z^2y, \qquad q\in L^1(\br_+), \quad z\in\bc_+, 
\end{equation}
where the potential $q$ may be complex-valued. There exists a unique pair of solutions $e_\pm(\cdot,z;q)$ of \eqref{difeqr+},
such that $e_\pm(x,\cdot;q)$ are analytic on the upper half-plane $\bc_+$, and
\begin{equation}\label{asymjos}
\begin{split}
e_+(x,z;q) &=e^{ixz}(1+o(1)), \qquad e_+'(x,z;q) = iz e^{ixz}(1 +o(1)), \\
e_-(x,z;q) &=e^{-ixz}(1+o(1)), \quad \ e_-'(x,z;q) = -iz e^{-ixz}(1+o(1)),
\end{split}
\end{equation}
as $x\to +\infty$, uniformly on compact subsets of $\bc_+$ (see, e.g., \cite[Sections 2.2 and 2.3]{Naimark}).
The Wronskian satisfies
\begin{equation}\label{wronhalf}
W(z,q)=W(e_+, e_-)=-2iz.
\end{equation}
Recall that $H = H_q$ denotes the Schr\"odinger--Dirichlet operator on $L^2(\br_+)$.

\subsection{Case of the real line}\label{subsec:full}
Consider the following differential equation on the real line $\br$
\begin{equation}\label{difeqr}
\pzch[y]:=-y''+\pzcq(x)y=z^2y, \qquad \pzcq\in L^1(\br), \quad z\in\bc_+, 
\end{equation}
where the potential $\pzcq$ may be complex-valued.

The result below is likely to be well known. We provide the proof for the sake of completeness.

\begin{proposition}\label{line}
There exists a unique pair of solutions $\pzce_\pm(\cdot,z;\pzcq)$ of \eqref{difeqr}, known as the Jost solutions, such that
$\pzce_\pm(x,\cdot;\pzcq)$ are analytic on the upper half-plane $\bc_+$, 
\begin{equation}
  \label{asymjosr1}
  \pzce_+(x,z;\pzcq) =e^{ izx}(1+o(1)), \quad \pzce_+'(x,z;\pzcq) = ize^{izx}(1+o(1))
\end{equation}
as $x \to + \infty$, and
\begin{equation}
  \label{asymjosr2}
  \pzce_-(x,z;\pzcq) =e^{-izx}(1+o(1)), \quad \pzce_-'(x,z;\pzcq) = - ize^{-izx}(1+o(1))
\end{equation}
as $x \to - \infty$, uniformly on compact subsets of $\bc_+$.

$\lambda=z^2$ is the eigenvalue of the corresponding Schr\"odinger operator $\ch_\pzcq$ on $L^2(\br)$ if and only if $\pzce_+$ and $\pzce_-$
are proportional, that is, the Wronskian
\begin{equation*} 
W(z, \pzcq):=\pzce_+(0,z;\pzcq)\pzce'_-(0,z;\pzcq)-\pzce_-(0,z;\pzcq)\pzce'_+(0,z;\pzcq)=0. 
\end{equation*}
The algebraic multiplicity $\nu(\lambda,\ch_\pzcq)$ of the eigenvalue $\l=z^2$ equals the multiplicity of the corresponding zero of $W(\cdot, \pzcq)$.
\end{proposition}
\begin{proof}
  The first statement, regarding the existence and analytic properties of the Jost solutions, may be seen by extending appropriate Jost solutions on the half-line.
  Indeed, let $s(x,z)$ and $c(x,z)$ denote the solutions of (\ref{difeqr}) such that
  \begin{equation*}
    s(0,z) = c'(0,z) = 0, \qquad s'(0,z) = c(0,z) = 1.
  \end{equation*}
  We define
  \begin{align*}
    \pzce_+(x,z;\pzcq) &= c(x,z)e_+(0,z;q_+) + s(x,z)e_+'(0,z;q_+) \\
   \pzce_-(x,z;\pzcq) &= c(x,z)e_+(0,z;q_-) - s(x,z)e_+'(0,z;q_-),
  \end{align*}
  where $q_\pm$ are potentials on the half-line such that
  \begin{equation*}
    q_\pm(x) := \pzcq(\pm x), \qquad x \in \br_+.
  \end{equation*}
  Notice that the functions $\pzce_\pm(\pm x,z;\pzcq)$, $x \in \br_+$, solve the Schr\"odinger  equations \eqref{difeqr+} with $q = q_\pm$.
  By computing the boundary conditions of $\pzce_\pm(\pm x,z;\pzcq)$ at $x = 0$, we see that   
  \begin{align*}
\pzce_+(x,z;\pzcq) &=
                       e_+(x,z;q_+),\qquad \ \ x\in\br_+, \\
    \pzce_-(x,z;\pzcq) &=
e_+(-x,z;q_-),\qquad x\in\br_-.
\end{align*}
The asymptotic relations \eqref{asymjosr1} and \eqref{asymjosr2} follow. 
The analyticity statement follows from the fact that
 $s(x,\cdot)$ and $c(x,\cdot)$ are entire functions (see, for instance, \cite[Lemma 5.7]{Teschl}) as well as the 
 analyticity of  $e_+(0,\cdot;q_\pm)$ and $e'_+(0,\cdot;q_\pm)$ on $\bc_+$.  

Next, we prove the second statement, characterising the eigenvalues of $\ch_\pzcq$. If the Jost solutions $\pzce_\pm$ are proportional, the eigenfunction
exists, and so $z^2$ is the eigenvalue. Conversely, assume that $\pzce_+$ and $\pzce_-$ are linearly independent.
The limit case on each half-line (cf. \eqref{asymjos}) means that $\pzce_\pm\notin L^2(\br_\mp)$. Hence, all solutions of \eqref{difeqr} from $L^2(\br_\pm)$
are of the form $c_\pm\,\pzce_\pm$. If $z^2\in\s_d(\ch_\pzcq)$, there is a solution $\pzce\in L^2(\br)$ of \eqref{difeqr} with
\begin{equation*}
\pzce(x,z;\pzcq)=
\begin{cases} c_+\pzce_+(x,z;\pzcq),\ & x\in\br_+, \\
c_-\pzce_-(x,z;\pzcq), \ & x\in\br_-,
\end{cases}
\end{equation*}
and so $\pzce_+$ and $\pzce_-$ are proportional. A contradiction completes the proof.

The final statement follows from \cite[Theorem 28]{LatSuk}.
\end{proof}

In what follows, we shall suppress indication of $z$ dependence where appropriate.
\subsubsection*{Compactly supported potentials}
Assume that $\pzcq$ is compactly supported, \newline $\text{supp}\,\pzcq\subset [-a,a]$,\ $a>0$. Then
\begin{equation}\label{freeasym}
\pzce_-(x,\pzcq)=e^{-izx}, \qquad \pzce_-'(x,\pzcq)= -iz\,e^{- izx}, \qquad  x\leq - a.
\end{equation}
Also, there exist $A_\pm(z)$ such that
\begin{equation}\label{oppos}
\begin{split}
\pzce_+(x,\pzcq) &= A_+(z)e^{izx}+A_-(z)e^{-izx}, \\
\pzce_+'(x,\pzcq) &= iz\Bigl(A_+(z)e^{izx}-A_-(z)e^{-izx}\Bigr), \quad x\le -a.
\end{split}
\end{equation}
We can easily calculate the Wronskian. For $x\le -a$,
\begin{equation*}
\begin{split}
W(\pzce_+, \pzce_-) &=\Bigl(A_+(z)e^{izx}+A_-(z)e^{-izx}\Bigr)\Bigl(-ize^{-izx}\Bigr) \\
&-iz\Bigl(A_+(z)e^{izx}-A_-(z)e^{-izx}\Bigr)e^{-izx} =-2izA_+(z)
\end{split}
\end{equation*}
and so
\begin{equation}\label{wrocomp}
W(z,\pzcq)=W(\pzce_+, \pzce_-)=-2izA_+(z).
\end{equation}
Note that equations analogous to \eqref{freeasym}, \eqref{oppos} and \eqref{wrocomp} also hold for the opposite half-line $x \geq a$.

\subsubsection*{Shifted potentials}
Next, consider a shifted equation
\begin{equation}\label{difeqrsh}
\pzch_X[y]:=-y''+\pzcq(x-X)y=z^2y, \qquad X>0. 
\end{equation}
All its solutions are shifts of the corresponding solutions of \eqref{difeqr}. In particular, the Jost solutions satisfy
\begin{equation}\label{jostsolsh}
\pzce_\pm(x,\pzcq(\cdot-X))=e^{\pm izX}\,\pzce_\pm(x-X,\pzcq).
\end{equation}
\subsubsection*{Symmetrisation of potentials}

The following result will allow us to apply the lower bounds of Section \ref{sec:db} to even extensions of dissipative barrier potentials.
We mentioned it in the introduction, see (\ref{eq:R-vs-R+}).

\begin{proposition}\label{evenext}
Given a potential $q\in L^1(\br_+)$, let $\pzcq_e$ be its even extension on the line
\begin{equation*}
\pzcq_e(-x)=\pzcq_e(x), \quad x\in\br; \qquad \pzcq_e\vert_{ \br_+}=q.
\end{equation*}
Then $\s_d(H_q)\subset \s_d(\ch_{\pzcq_e})$, and moreover, for each $\lambda\in \s_d(H_q)$,
\begin{equation}\label{evenmult}
\nu(\lambda, H_q)\leq \nu(\lambda, \ch_{\pzcq_e}).
\end{equation}
\end{proposition}
\begin{proof}
It is clear from the definition, that
\begin{equation*}
\pzce_-(x,z;\pzcq_e)=\pzce_+(-x,z;\pzcq_e), \quad \pzce_-'(x,z;\pzcq_e)=-\pzce_+'(-x,z;\pzcq_e), \qquad x\in\br.
\end{equation*}
Hence, $W(z,\pzcq_e)=-2\pzce_+(0,z;\pzcq_e)\,\pzce_+'(0,z;\pzcq_e)$. 
But $\pzcq_e\vert_{\br_+}=q$, so
\begin{equation*}
\pzce_+(x,z;\pzcq_e)=e_+(x,z;q), \quad x\in\br_+; \quad W(z,\pzcq_e)=-2e_+(0,z;q)\,e_+'(0,z;q).
\end{equation*}
The result now follows from Proposition \ref{line}.
\end{proof}

\subsection{Auxillary spectral approximation results}\label{subsec:approx}
\subsubsection*{Large shifts}

The following lemma and its corollary are crucial for the proof of Theorem \ref{mainth}. A more general, but slightly less 
precise, version of this result has been proven in \cite[Lemma 4]{BogAcc} by invoking the abstract notion of limiting essential spectrum (cf. \cite{BogEss}).
In contrast to that result, it is important for us to account for algebraic multiplicities, and our proof only relies on 
basic ODE theory and complex analysis.

\begin{lemma}\label{spinc}
Let $q\in L^1(\br_+)$ and $\pzcq\in L^1(\br)$ be potentials with compact supports. For any  $X>0$, denote 
\begin{equation*}
q(x,X):=q(x)+\pzcq(x-X), \qquad x \in \br_+.
\end{equation*}
Then $q(\cdot,X) \in L^1(\br_+)$ for all $X > 0$, and 
\begin{equation}\label{product}
\lim_{X\to\infty}e_+(0,z;q(\cdot,X))=-\frac{e_+(0,z;q)\,W(z,\pzcq)}{2iz}=e_+(0,z;q)\frac{W(z,\pzcq)}{W(z,q)}
\end{equation}
uniformly on compact subsets of $\bc_+$.
\end{lemma}
\begin{proof}
Assume that
\begin{equation*}
\text{supp}\,q\subset [0,b], \qquad \text{supp}\,\pzcq\subset [-a,a], \qquad a,b>0,
\end{equation*}
so that $\text{supp}\,\pzcq(\cdot-X)\subset [X-a, X+a]$. Assume also that $X$ is so large that
\begin{equation*}
b<\frac{X}2:=Y<X-a.
\end{equation*}
Then $\text{supp}\,\pzcq(\cdot-X)\subset\br_+$, and the supports of $q$ and $\pzcq(\cdot-X)$ are disjoint. For the Jost solution, we have
\begin{equation}\label{josthalf}
e_+(x,q(\cdot,X))=
\begin{cases} c_+ e_+(x,q)+c_-e_-(x,q), & 0\le x\le Y \\
\pzce_+(x,\pzcq(\cdot-X))=e^{izX}\pzce_+(x-X,\pzcq),  & x>Y,
\end{cases}
\end{equation}
for some $c_\pm=c_\pm(X,z) \in \bc$. The adjustment conditions at $Y$ yield
\begin{equation*}
\begin{split}
c_+ e_+(Y,q) + c_- e_-(Y,q) &=e^{izX}\,\pzce_+(-Y,\pzcq), \\
c_+ e_+'(Y,q) + c_- e_-'(Y,q) &=e^{izX}\,\pzce_+'(-Y,\pzcq),
\end{split}
\end{equation*}
or, in matrix form,
\begin{equation*}
\begin{bmatrix}
e_+(Y,q) & e_-(Y,q) \\
e_+'(Y,q) & e_-'(Y,q)
\end{bmatrix}
\begin{bmatrix}
c_+ \\
c_-
\end{bmatrix}=e^{izX}
\begin{bmatrix}
\pzce_+(-Y,\pzcq) \\
\pzce_+'(-Y,\pzcq)
\end{bmatrix}.
\end{equation*}
A matrix inversion yields
\begin{equation*}
\begin{bmatrix}
c_+ \\
c_-
\end{bmatrix}=\frac{e^{izX}}{W(z,q)}\,
\begin{bmatrix}
e_-'(Y,q) & -e_-(Y,q) \\
-e_+'(Y,q) & e_+(Y,q)
\end{bmatrix}
\begin{bmatrix}
\pzce_+(-Y,\pzcq) \\
\pzce_+'(-Y,\pzcq)
\end{bmatrix}.
\end{equation*}

We can now calculate the Jost function from the upper relation in \eqref{josthalf},
taking into account \eqref{asymjos} and \eqref{wronhalf}
\begin{equation*}
\begin{split}
&{} e_+(0,q(\cdot,X)) =c_+e_+(0,q)+c_-e_-(0,q) \\
&=-\frac{e^{izY}}{2iz}
\begin{bmatrix}
e_+(0,q) & e_-(0,q)
\end{bmatrix}
\begin{bmatrix}
-iz+o(1) & -1+o(1) \\
e^{izX}(-iz+o(1)) & e^{izX}(1+o(1))
\end{bmatrix}
\begin{bmatrix}
\pzce_+(-Y,\pzcq) \\
\pzce_+'(-Y,\pzcq)
\end{bmatrix} \\
&=-\frac{e^{izY}}{2iz}
\begin{bmatrix}
e_+(0,q) & e_-(0,q)
\end{bmatrix}
\begin{bmatrix}
f_+(X,\pzcq) \\
f_-(X,\pzcq)
\end{bmatrix},
\end{split}
\end{equation*}
where
\begin{equation*}
\begin{split}
f_+(X,\pzcq) &:= (-iz+o(1))\pzce_+(-Y,\pzcq)+(-1+o(1))\pzce_+'(-Y,\pzcq), \\
f_-(X,\pzcq) &:= e^{izX}\left\{(-iz+o(1))\pzce_+(-Y,\pzcq)+(1+o(1))\pzce_+'(-Y,\pzcq)\right\}, \ Y=\frac{X}2.
\end{split}
\end{equation*}

Since $Y>a$, then, by \eqref{oppos},
\begin{equation*}
\begin{split}
e^{izY}f_+(X, \pzcq) &=(-iz+o(1))\bigl(A_+ +A_-e^{izX}\bigr)+(-iz+o(1))\bigl(A_+ -A_-e^{izX}\bigr) \\
&=-2iz A_+ +o(1), \qquad X\to\infty,
\end{split}
\end{equation*}
uniformly on compact subsets of $\bc_+$. It is clear from \eqref{oppos}, that
\begin{equation*}
e^{izY}f_-(X, \pzcq)=o(1), \qquad X\to\infty,
\end{equation*}
also uniformly on compact subsets of $\bc_+$.
The relation \eqref{wrocomp} completes the proof.
\end{proof}

Before we move on, let us clarify what we shall mean by a collection of eigenvalues. When we say that there exists a collection of $N \in \bn$  eigenvalues 
$\lambda_1,...,\lambda_N$ of an operator $T$, we mean that:
 \begin{enumerate}
 \item  $\lambda_j$ is an eigenvalue of $T$ for each $j \in \{1,...,N\}$, and
 \item if $\lambda$ is repeated $\nu$ times in the collection $\lambda_1,...,\lambda_N$, then $\lambda$ is an eigenvalue of $T$ with algebraic 
multiplicity at least $\nu$.
 \end{enumerate}

An integer-valued function $\nu(\cdot,T)$ is said to be an algebraic multiplicity with respect to a linear operator $T$, if $\nu(\lambda,T)$ equals
the algebraic multiplicity of $\lambda$ in case when $\lambda\in\s_d(T)$, and $\nu(\lambda,T)=0$ otherwise.

\begin{corollary}\label{approx}
Let the potentials $q$ and $\pzcq$ be defined as above. Given $\lambda\in\bc\bsl\br_+$, put
\begin{equation}\label{multpoint}
\nu=\nu(\lambda):=\nu(\lambda, H_q)+\nu(\lambda, \ch_{\pzcq}).
\end{equation}
Then $\l\in \s_d(H_q)\cup\s_d(\ch_\pzcq)$, if and only if there exists a collection of $\nu$ eigenvalues  
$\l_X^{(1)}$, ..., $\l_X^{(\nu)}$ of $H_{q(\cdot,X)}$, $X > 0$ large enough,  such that
\begin{equation*}
\lim_{X\to\infty} \lambda_X^{(j)}=\lambda, \qquad j=1,2,\ldots,\nu.
\end{equation*}
\end{corollary}
\begin{proof}
 By Proposition \ref{line} (and similar property of the Jost function $e_+(0,\cdot;q)$), $\l=z^2\in\s_d(H_q)\cup\s_d(\ch_\pzcq)$ if and only if
$z\in\bc_+$ is a root of the right-hand side \eqref{product} with multiplicity equal to $\nu(\l)$ \eqref{multpoint}. The rest is a direct consequence 
of Lemma \ref{spinc} and Hurwitz's theorem.
\end{proof}
In particular, note that if $\nu(\mu, H_q)=\nu(\mu, \ch_\pzcq)=0$, then $\mu$ is separated from the discrete spectra $\s_d(H_{q(\cdot,X)})$ for all large enough $X$.

\subsubsection*{Truncation}
Given a potential $q\in L^1(\br_+)$, we define its truncation at the level $X>0$ as
\begin{equation}\label{trun}
q_X(x):=
\begin{cases} q(x),\ & 0\le x\le X; \\
0, \ & x>X.
\end{cases}
\end{equation}
Let $(X_n)_{n\in\bn}$ be a sequence of positive numbers such that $\lim_{n \to \infty} X_n = \infty$. Put 
\begin{equation*} 
q_n:=q_{X_n}, \qquad  H_n:= H_{q_n}. 
\end{equation*}

\begin{lemma}\label{trunc}
In the above notation, the limit relation
\begin{equation}\label{limtrun}
\lim_{X\to\infty} e_+(0,z;q_X)=e_+(0,z;q)
\end{equation}
holds uniformly on compact subsets of $\bc_+$. In particular, $\l\in\bc\bsl\br_+$ is an eigenvalue of $H=H_q$ of algebraic multiplicity $\nu$
if and only if there exists a collection of eigenvalues $\l_n^{(1)},...,\l_n^{(\nu)}$  of $H_n$ such that
\begin{equation*}
\lim_{n\to\infty} \l_n^{(j)}=\l, \qquad j=1,2,\ldots,\nu. 
\end{equation*}
  
\end{lemma}
\begin{proof}
The argument is similar to one above. We have
\begin{equation*}
e_+(x,q_X)=
\begin{cases} c_+e_+(x,q)+c_-e_-(x,q),\ & 0\le x< X; \\
e^{izx}, \ & x\ge X,
\end{cases}
\end{equation*}
$c_\pm=c_\pm(X,z)$. The adjustment conditions at $X$ yield
\begin{equation*}
\begin{split}
c_+e_+(X,q)+c_-e_-(X,q) &= e^{izX}, \\
c_+e_+'(X,q)+c_-e_-'(X,q) &=iz\,e^{izX},
\end{split}
\end{equation*}
or in matrix form
\begin{equation*}
\begin{bmatrix}
e_+(X,q) & e_-(X,q) \\
e_+'(X,q) & e_-'(X,q)
\end{bmatrix}
\begin{bmatrix}
c_+ \\
c_-
\end{bmatrix}=e^{izX}\begin{bmatrix}
1 \\ iz
\end{bmatrix}
\end{equation*}

The matrix inversion gives
\begin{equation*}
\begin{bmatrix}
c_+ \\
c_-
\end{bmatrix}=-\frac{e^{izX}}{2iz}
\begin{bmatrix}
e_-'(X,q) & -e_-(X,q) \\
-e_+'(X,q) & e_+(X,q)
\end{bmatrix}
\begin{bmatrix}
1 \\ iz
\end{bmatrix},
\end{equation*}
and so
\begin{equation*}
\begin{split}
c_+(X,z) &=-\frac{e^{izX}}{2iz}\Bigl[e_-'(X,q)-iz e_-(X,q)\Bigr], \\
c_-(X,z) &=-\frac{e^{izX}}{2iz}\Bigl[-e_+'(X,q)+iz e_+(X,q)\Bigr].
\end{split}
\end{equation*}
Finally,
\begin{equation*}
\begin{split}
&{} e_+(0,q_X)= \\
&-\frac{e^{izX}}{2iz}\left\{\Bigl[e_-'(X,q)-iz e_-(X,q)\Bigr]e_+(0,q)+\Bigl[-e_+'(X,q)+iz e_+(X,q)\Bigr]e_-(0,q)\right\},
\end{split}
\end{equation*}
and \eqref{limtrun} follows from the asymptotic relations \eqref{asymjos}.

The second statement is clear thanks to Hurwitz's theorem.
\end{proof}

\subsection{Main result}\label{subsec:main}

We are in a position now to prove the main result of the section. 

\begin{theorem}\label{mainth}
There exists a potential $q_\infty\in L^1(\br_+)$ with infinite Jensen sum.
\end{theorem}
\begin{proof}
Let $(\gamma_n)_{n \in \bn},(R_n)_{n \in \bn},(X_n)_{n \in \bn} \subset \br_+$, to be further specified. Define a sequence of Schr\"odinger
operators on the line
\begin{equation}\label{evendb}
\cl_n y:=-y''+\pzcl_n y, \qquad \pzcl_n(x):=i\gamma_n\chi_{[-R_n, R_n]}(x)\in L^1(\br), \quad n\in\bn.
\end{equation}

Let $(N_n)_{n \in \bn_0}$ be defined such that $N_0 = 0$ and, for $n \geq 1$, $N_n - N_{n-1}$ equals the number of eigenvalues of $\cl_n$,
counting algebraic multiplicity. We place all the eigenvalues $(\l_j)_{j\in\bn}$ of all operators $\cl_n$ in a single sequence in such a way that
\begin{equation*}
\{\l_{N_{n-1}+1}, \ldots, \l_{N_n}\}=\sigma_d(\cl_n), \qquad n\in\bn.
\end{equation*}

Define consecutively a sequence of potentials
\begin{equation*}
q_n(x):=q_{n-1}(x)+i\g_n\chi_{[X_n, X_n+2R_n]}(x)=q_{n-1}(x)+\pzcl_n(x-X_n-R_n), \quad n\in\bn,
\end{equation*}
$q_0\equiv0$, or, in other words,
\begin{equation}\label{partsum}
q_n(x)=\sum_{k=1}^n i\g_k\chi_{[X_k, X_k+2R_k]}(x).
\end{equation}
We assume that $X_{k+1}>X_k+2R_k$, so the intervals $[X_k, X_k+2R_k]$, $k\in\bn$, are disjoint.

Let $M_n$ denote the cardinality of the discrete spectrum $\s_d(H_{q_n})$, counting algebraic multiplicity
\begin{equation*}
\sigma_d(H_{q_n})=\{\lambda_{j,n}\}_{j=1}^{M_n}.
\end{equation*}
In view of Corollary \ref{approx}, we see that for large enough $X_n$, 
\begin{equation*}
M_{n-1}+N_n-N_{n-1}\le M_n, \qquad N_n-N_{n-1}\le M_n-M_{n-1},
\end{equation*}
and, as $M_0=N_0=0$, it follows $N_n\le M_n$ for all $n\in\bn$.

By Corollary \ref{approx}, for each $n \in \bn$, we can set $X_n$ large enough such that the collection of eigenvalues 
$\lambda_{j,n}$, $j = 1,...,N_n$, of $H_{q_n}$  (note that $N_n\le M_n$) satisfy
\begin{equation}\label{app1}
\begin{split}
|\l_{j,n}-\l_j| &+|\im\sqrt{\l_{j,n}}-\im\sqrt{\l_{j}}| \le \frac3{(\pi n)^2}\,\im\sqrt{\l_{j}}, \\ 
j &=N_{n-1}+1,\ldots, N_n, \quad n\in\bn,
\end{split}
\end{equation}
and
\begin{equation}\label{app2}
\begin{split}
|\l_{j,n}-\l_{j,n-1}| &+|\im\sqrt{\l_{j,n}}-\im\sqrt{\l_{j,n-1}}| \le \frac3{(\pi n)^2}\,\im\sqrt{\l_{j}}, \\ 
j &=1,\ldots, N_{n-1}, \quad n=2,3,\ldots.
\end{split}
\end{equation}

For each fixed $j\in\bn$, $\lambda_{j,n}$ exists for all $n \geq m$, where $m \in \bn$ is such that $\lambda_j \in \sigma_d(\cl_m)$.
The sequence $(\l_{j,n})_{n\geq m}$ is Cauchy, so there exists
\begin{equation*}
\mu_j:=\lim_{n\to\infty} \l_{j,n}.
\end{equation*}
Next, putting $\l_{j,m-1}:=\l_j$, we have  for any $k \geq m+1$
\begin{equation*}
\sum_{n=m}^k \Bigl(\im\sqrt{\l_{j,n}}-\im\sqrt{\l_{j,n-1}}\Bigr)=\im\sqrt{\l_{j,k}}-\im\sqrt{\l_{j}},
\end{equation*}
so
\begin{equation*}
\begin{split}
&{}\bigl|\im\sqrt{\l_{j,k}}-\im\sqrt{\l_{j}}\bigr| \le \sum_{n=m}^k \bigl|\im\sqrt{\l_{j,n}}-\im\sqrt{\l_{j,n-1}}\bigr| \\
&=\bigl|\im\sqrt{\l_{j,m}}-\im\sqrt{\l_{j}}\bigr|+\sum_{n=m+1}^k \bigl|\im\sqrt{\l_{j,n}}-\im\sqrt{\l_{j,n-1}}\bigr| \\
&\le \frac{3\,\im\sqrt{\l_{j}}}{\pi^2}\,\sum_{n=1}^\infty \frac1{n^2}=\frac12\,\im\sqrt{\l_{j}},
\end{split}
\end{equation*}
whence it follows, as $k\to\infty$, that
\begin{equation}\label{below}
\im\sqrt{\mu_j}\ge\frac12\,\im\sqrt{\l_j}, \qquad j\in\bn,
\end{equation}
and in particular, $\mu_j\in\bc\bsl\br_+$.

Set
\begin{equation}\label{eq:gam_n-R_n}
\begin{split}
\gamma_{n} &= \frac{1}{(n \log^2(n+2))^4}<1,  \\ 
R_{n} &= 1200\,\gamma_{n}^{-3/4}=1200\,(n\log^2(n+2))^3, \quad n\in\bn.
\end{split} 
\end{equation}
Define a potential on $\br_+$
\begin{equation}
q_\infty:=  \sum_{n=1}^\infty i \gamma_n \chi_{[X_n,X_n+2R_n]}. 
\end{equation}
Then,
\begin{equation*}
     \| q_\infty\|_1 = 2\sum_{n=1}^\infty \gamma_n R_n = 2400 \sum_{n=1}^\infty \frac{1}{n \log^2(n+2)} < \infty,
\end{equation*}
so $q_\infty \in L^1(\br_+)$.

The partial sums \eqref{partsum} can be viewed as truncations of $q_\infty$ at the level $X_n+2R_n$. Lemma \ref{trunc} implies that
for each $j \in \bn$, $\mu_j$ is an eigenvalue of $H_{q_\infty}$ with algebraic multiplicity greater than or equal to the
number of times it appears in the sequence $(\mu_k)_{k \in \bn}$. It follows that
\begin{equation*}
J(H_{q_\infty}) \geq \sum_{j=1}^\infty \im \sqrt{\mu_j} \geq \frac{1}{2}\sum_{j=1}^\infty \im \sqrt{\lambda_j} = 
\frac{1}{2} \sum_{n=1}^\infty J(\cl_n).
\end{equation*}

Recall that $L_n:=L_{\gamma_n,R_n}$ is defined in \eqref{eq:db-defn-intro} as the Schr\"odinger operator on $L^2(\br_+)$ with potential 
$i\gamma_n\chi_{[0,R_n]}$. By Proposition \ref{evenext}, any eigenvalue of $L_n$ is also an eigenvalue of $\cl_n$, and \eqref{evenmult} holds. 
Hence, employing the left inequality in Proposition \ref{prop:two-sided}, we have
\begin{equation}
\label{eq:J-LgR-lower}
J(\cl_n) \geq J(L_n) \geq \frac{1}{32 \pi} \gamma_n R_n \log R_n, \qquad R_n \geq 600(\gamma_n^{3/4} + \gamma_n^{-3/4}).
\end{equation}
The latter inequality is true for all $n\in\bn$ due to the choice of $R_n$ \eqref{eq:gam_n-R_n} and $\gamma_n<1$. Consequently,
\begin{equation}\label{eq:J-Hq-lowerbound}
J(H_{q_\infty}) \geq \frac{1}{64 \pi} \sum_{n=1}^\infty \gamma_n R_n \log R_n = \frac{600}{32\pi} \sum_{n=1}^\infty \frac{\log R_n}{n \log^2(n+2)}.
\end{equation}

 Since $\log R_n \sim 3 \log n$ as $n \to \infty$, the sum on the right hand side of (\ref{eq:J-Hq-lowerbound}) diverges. 
We conclude that the Jensen sum $J(H_{q_\infty})=\infty$, completing the proof.
\end{proof}

\bibliographystyle{plain}
\bibliography{Citations.bib}

\end{document}